\PassOptionsToPackage{table}{xcolor}
\documentclass[10pt]{article}

\usepackage{tikz} 
\usetikzlibrary{decorations.pathreplacing}
\usepackage{array}
\usepackage[a4paper,vmargin=2.5cm,hmargin=2.5cm ]{geometry}

\usepackage{amsmath,amsbsy,amssymb,amsthm}
\newtheorem{customassumption}{Assumption}

\usepackage{algorithm}
\usepackage{algpseudocode}
\usepackage{changes}
\usepackage{longtable}
\usepackage{cite}
\usepackage{stackengine}
\usepackage{bm}
\usepackage{subfigure}
\usepackage{subfigmat}
\usepackage{soul}
\usepackage{amsmath}
\usepackage{cancel}
\usepackage{booktabs}
\usepackage{multirow}
\usepackage{lipsum} 

\usepackage{graphicx}
\usepackage{epstopdf}
\usepackage{color}
\usepackage[table]{xcolor}
\definecolor{lightblue}{rgb}{0.95, 0.975, 1.0}
\usepackage{amsmath}

\usepackage{xcolor} 

\usepackage{float}
\usepackage{changepage} 
\usepackage{hyperref}
\usepackage{mathpazo} 
\usepackage{lettrine}
\usepackage{tikz}
\usepackage{amsmath}

\usepackage{graphicx}   

\usetikzlibrary{decorations.pathreplacing, calc}

\usepackage{pifont}

\colorlet{mdtRed}{blue!50!black}
\hypersetup{pdfpagemode={UseOutlines},
	bookmarksopen=true,
	bookmarksopenlevel=0,
	hypertexnames=false,
	colorlinks=true,
	citecolor=magenta,
	linkcolor=red,
	urlcolor=mdtRed,
	pdfstartview={FitV},
	unicode,
	breaklinks=true,
}

\usepackage[table]{xcolor} 
\usepackage{graphicx}      
\usepackage{array}            
\usepackage{caption}          
\usepackage{lscape} 
\usepackage{colortbl} 
\newtheorem{assumption}{Assumption}[section]

\newtheorem{theorem}{Theorem}[section]

\begin{document}
\begin{center}
	{\fontsize{15}{24}\selectfont \textbf{A Surrogate Framework for General Cruise-Phase Optimization of Commercial Aircraft}}
\end{center}
\textbf{Amin Jafarimoghaddam}$^\star$\footnote{$^\star$Corresponding Author.}, \textbf{Manuel Soler}$^{\dagger}$, \textbf{María Cerezo-Magaña} $^{\dagger}$\\	
$^{\star,\dagger}${\em{Department of Aerospace Engineering, Universidad Carlos III de Madrid. Avenida de la Universidad, 30, Leganes, 28911 Madrid, Spain
}}\footnote{Email addresses: ajafarim@pa.uc3m.es (A. Jafarimoghaddam), masolera@ing.uc3m.es (M. Soler), mcerezo@ing.uc3m.es (M. Cerezo)}	\\


\begin{center}
	\textbf{\large Abstract}
\end{center}
Optimizing commercial aircraft cruise trajectories using the Pontryagin Maximum Principle (PMP) is particularly challenging due to the nonlinear dynamics of aircraft speed, complex costate dynamics, and the inclusion of two continuous controls, one of which (thrust) is typically a singular, affine input. We present a surrogate optimization framework, accounting for space-dependent wind fields, flight-sensitive areas, and equality/inequality constraints. To simulate atmospheric wind, we propose a divergence-free composite inviscid flow model showing structural consistency with large-scale patterns observed in atmospheric wind datasets. Flight-sensitive areas are clustered as parametrized ellipses, considered as soft constraints. Building on turnpike property, the surrogate framework enforces quasi-steady aircraft speed on infinitesimal segments, thereby eliminating the thrust control and treating speed as a direct control. Comprehensive numerical experiments on various setups confirm that the surrogate framework reproduces optimal solutions identical to those of the original problem while significantly reducing computational time and implementation efforts. The accompanying \textbf{supplementary material} includes implementation details and open-source code (with interactive built-in and user-defined environments for aircraft and wind functions, and for flight-sensitive areas), together with an additional example that examines wind-induced variability in optimal solutions via Monte Carlo analysis enabled by the high-speed PMP-based surrogate framework.

\bigskip
\section{Introduction}\label{sec1}
Minimizing direct‐operating cost (DOC), which combines fuel burn and arrival time, remains a cornerstone of commercial aviation optimization. Commercial jets, however, must often reroute around flight‐sensitive regions, whether to avoid hazardous weather (e.g., thunderstorms \cite{1}) or to steer clear of designated climate‐sensitive zones. In particular, penetration of climate‐sensitive areas can trigger persistent contrail formation, with demonstrable impacts on global warming \cite{2,3}. To capture these operational constraints, we augment the DOC functional with penalty terms that penalize any trajectory segment traversing such regions.

In this study, we develop a general surrogate optimization framework grounded in Pontryagin’s Maximum Principle (PMP), also known as the indirect method, for optimizing the cruise phase of commercial aircraft. The objective function integrates both DOC and penalties associated with traversing flight-sensitive regions.

It is important to note that, despite extensive research on various optimization strategies for commercial aircraft trajectories, Pontryagin’s Maximum Principle (PMP) has received relatively limited attention in the literature \cite{4}. This may be due to the intricate and often sensitive nature of the resulting co-state system. The complexity further increases in the presence of discontinuous (or singular) controls, multiple control variables, and state inequality constraints; each of which can lead to complex junction criteria in the optimization system \cite{5}. Notably, these challenges are inherent characteristics of the general cruise-flight optimization problem.

To contextualize the contributions of this work, we begin with a concise overview of existing research that applies the PMP to the optimization of commercial aircraft cruise trajectories. This review serves to highlight the methodological novelty and broader scope of the present study.

A careful review of the literature reveals that PMP-based approaches to cruise-flight optimization generally fall into two main categories:

\textbf{1) One-dimensional models:} These simplify cruise flight dynamics to a single spatial dimension. Typical formulations involve two or three state variables, such as aircraft mass, speed, and occasionally the distance traveled. Control input is usually limited to throttle setting. Studies within this framework include \cite{5,6,7,9,10,11,12}. To elaborate, Jafarimoghaddam and Soler \cite{5,12} generally target flights in a vertical plane. Although the full vertical plane is resolved with two controls (flight path angle and throttle setting) considered as active controls in free-routing flight optimization, the cruise phase is effectively reduced to one dimension, where the lateral path (arising from horizontal-plane considerations) is ignored. Pargett and Ardema \cite{6} formulate the constant-altitude cruise range as a singular optimal control problem with mass and speed as states and throttle setting as the control, and derive the singular arc for the control problem. Rivas and Valenzuela \cite{7} address maximum-range cruise at constant altitude by considering the compressible drag model by Cavcar and Cavcar \cite{8}. Franco et al.\ \cite{9} analyze minimum-fuel constant-altitude cruise subject to a fixed arrival-time constraint as a singular optimal control problem. Franco and Rivas \cite{10} extend the fixed-time cruise analysis to include constant horizontal winds. Finally, Franco and Rivas \cite{11} solve climb--cruise--descent trajectories in a vertical plane. However, in each phase, they employ only one active control to formulate the optimal control problem, and subsequently use junction criteria to connect the flight phases.

\textbf{2) Two-dimensional models:} These consider horizontal plane navigation using two spatial coordinates. At a minimum, these models include aircraft position as state variables and use heading angle as the control. Some studies, such as Ng et al.\ \cite{14,15}, and Sridhar et al.\ \cite{13} also incorporate aircraft mass, but omit its influence on the co-state equations, specifically neglecting the mass derivative of the Hamiltonian. More recent contributions by Jafarimoghaddam and Soler \cite{16,Amin-zermelo} broaden the model to include both throttle and heading angle as active controls in a four/five-state system. However, they do not address key elements such as avoidance of flight-sensitive regions, Mach constraints, or compressibility effects. Furthermore, the PMP framework developed in \cite{16} and \cite{Amin-zermelo}, while theoretically complete, presents considerable challenges in terms of computational implementation. Specifically, these models require prior knowledge of the number and sequence of optimal arcs, and their performance could sometimes be sensitive to the initialization of switching times. Notably, a significant portion of the theoretical and computational complexity arises from the dynamics of the aircraft speed and the behavior of its associated co-state.

In this research, we present a unified, PMP-based surrogate framework for cruise-flight optimization that circumvents the theoretical and numerical difficulties of the full optimal-control formulation while preserving the optimal solutions of the original problem. The approach retains the full four-state cruise model but removes the singular throttle control by recasting it as an equivalent speed-control law, resulting in two active controls: aircraft speed and heading angle. Spatially varying winds are formulated via a composite inviscid flow model (well-suited for atmospheric winds) which can simply be calibrated to reanalysis data, and flight-sensitive regions are represented by general space-dependent cost or penalty fields. In addition, the surrogate framework naturally accommodates practical inequality constraints (e.g., Mach, and thrust). Consequently, the method produces a compact, physically consistent boundary-value problem that is substantially more tractable numerically, enabling accurate, high-fidelity trajectory optimization with substantially reduced computational time.

\section{Problem Statement}\label{sec2}
We adopt the following assumptions and modeling choices for the point-mass cruise dynamics of commercial aircraft: \textbf{a)} {flat Earth approximation;} for flight altitudes below approximately $30\,\text{km}$ and speeds under $800\,\text{m/s}$, the flat Earth model has been shown to be effectively equivalent to the spherical Earth model~\cite{17}, and \textbf{b)} {simplified heading dynamics;} in the context of commercial aircraft operations, multiple studies have shown that treating the heading angle as a direct control input, rather than a state variable, has a negligible impact on the optimal trajectory outcomes~\cite{18,2,3,16}. For a detailed justification, refer to \textbf{Appendix}\ \ref{app1}.

Under these assumptions, the cruise-phase dynamics are defined for all $t \in [0, t_f]$ as \cite{mythesis}:
\begin{equation}\label{eq1}
	\begin{split}
		&\frac{dx}{dt}=v\cos(\chi(t))+W_x(x,y)=:F_x,\\
		&\frac{dy}{dt}=v\sin(\chi(t))+W_y(x,y):=F_y,\\
		&\frac{dv}{dt}=\frac{\Pi(t) T_{max}(v,\bar{h})-D(m,v,\bar{h})}{m}:=F_v,\\
		&\frac{dm}{dt}=-\Pi(t) C_s(v,\bar{h})T_{max}(v,\bar{h}):=F_m,\\
		&L=mg,\\
		&x(0)=x_0,\quad y(0)=y_0,\quad v(0)=v_0,\quad m(0)=m_0,\\
		&x(t_f)=x_f,\quad y(t_f)=y_f,\quad v(t_f)=v_f.\\
	\end{split}
\end{equation}

The variables in the above equations are defined as follows: the coordinates $x$ and $y$ define a horizontal plane in which the aircraft's cruise flight occurs; $t$ represents time, with $t_f$ denoting the final time; $W_x$ and $W_y$ are the wind velocity components along the $x$ and $y$ directions, respectively; $v$ is the airspeed of the aircraft (relative to the wind); $\chi$ denotes the heading angle; $\Pi$ is the throttle setting; $T_{\text{max}}$ is the maximum available thrust; $D$ is the aerodynamic drag force; $m$ is the mass of the aircraft; $L$ is the lift force; $C_s$ is the specific fuel consumption; $\bar{h}$ is the constant altitude at which cruise flight occurs; and $g$ is the gravitational acceleration.

The parameters $T_{\text{max}}(v,\bar{h})$ and $C_s(v,\bar{h})$ are part of the propulsion model, while $D(m,v,\bar{h})$ characterizes the compressible aerodynamic drag polar. These functions are detailed in \textbf{Appendix} \ref{app2}.

The controls $\Pi(t)$ and $\chi(t)$, are subject to the following inequality constraints, $\forall t\in[0,t_f]$:
\begin{equation}\label{eq1a}
	\begin{split}
		&\chi_{min}\leq\chi(t)\leq\chi_{max},\qquad \Pi_{min}\leq\Pi(t)\leq\Pi_{max}.
	\end{split}
\end{equation}
Furthermore, the aircraft flight is typically constrained by Mach number ($M:=v(c_0\bar{R}\Theta)^{-\frac{1}{2}}$):
\begin{equation}\label{eq1b}
	\begin{split}
		&M_{min}\leq M(t)\leq M_{max}.
	\end{split}
\end{equation} 
 
Considering the dynamic and algebraic constraints, the objective is to minimize both total flight time and fuel consumption while avoiding designated flight-sensitive areas.

These sensitive areas are typically those with high probability of hazardous phenomena such as thunderstorms, and persistent contrails.
Integrating these factors, an objective function can be written as:
\begin{equation}\label{eq1cc}
	\begin{split}
		&\min_{\Pi(t),\chi(t),t_f}\mathcal{J}:=\mathbf{c}_tt_f+\mathbf{c}_mm_f+\int_{0}^{t_f}g(x,y)dt.
	\end{split}
\end{equation} 

where $g(x,y)$ is an imposed penalty function for the sensitive areas; and $\mathbf{c}_t$, and $\mathbf{c}_m$ are coefficients, representing the trade-off among the objective terms. 
\subsection{Wind Model}\label{subsec21}

In theory, atmospheric wind is perceived as a turbulent, incompressible flow governed by the Navier-Stokes equations. Under this framework, the momentum equations become highly complex, featuring eddies spanning a wide range of length scales, from synoptic (hundreds of kilometers) down to inertial subranges (millimeters and smaller). However, the incompressible continuity condition (commonly written as \(\nabla\cdot\mathbf{u}=0\)) retains exactly the same form in both laminar and turbulent regimes.

Because the velocity components are inherently coupled through the Navier–Stokes system, any wind‐field model that aims to be both physically consistent and data‐informed should go beyond purely empirical curve fits to historical or reanalysis data (e.g., ERA5). In particular, one must respect incompressibility and, at least approximately, the momentum balance in the inviscid (large‐scale) limit. To meet these requirements, we propose a composite inviscid wind model built as a superposition of elementary flow primitives: uniform flow, regularized point vortices, potential‐flow dipoles, and regularized point sources and sinks. Each primitive is an exact solution of the incompressible Euler equations (i.e., the Navier–Stokes equations with zero viscosity), so that the composite model automatically satisfies \(\nabla\cdot\mathbf{u}=0\) and yields inviscid solutions to the momentum equations, capturing the dominant large‐scale patterns of the flow.

Therefore, by constructing a divergence‐free, inviscid wind field from analytic flow primitives, we enforce the most fundamental physical constraints, namely incompressibility \(\nabla\cdot\mathbf{u}=0\) and momentum balance in the Euler limit. In addition, because each fitting parameter corresponds to a fluid‐mechanical feature, the model yields realistic large‐scale wind patterns with a small number of parameters. Consequently, stochastic simulations based on this composite inviscid model can represent wind variability in a more physically grounded manner and with fewer samples than purely data‐driven techniques (such as polynomial fitting, spline fitting and neural network operators) whose coefficients may lack an obvious connection to fluid-mechanical processes.
     

\subsubsection{Composite Wind‐Field Construction}
Let $(x,y)\in\mathbb{R}^2$.  We define the total wind‐field as:
\begin{equation}
	\mathbf W(x,y)=
	\bigl(W_x,\,W_y\bigr)=
	\underbrace{U_\infty\,\mathbf e_x + V_\infty\,\mathbf e_y}_{\text{uniform flow}}
	+
	\underbrace{\sum_{k=1}^{M_v}\mathbf W^{(v)}_k(x,y)}_{\text{vortex flow}}
    +
	\underbrace{\sum_{\ell=1}^{M_d}\mathbf W^{(d)}_\ell(x,y)}_{\text{dipole flow}}
	+
	\underbrace{\sum_{j=1}^{M_s}\mathbf W^{(s)}_j(x,y)}_{\text{source/sink flow}}.
	\label{eq:composite}
\end{equation}
where $M_v,M_d$, and $M_s$, are the number of introduced vortices, dipoles, and sources/sinks respectively. The primitives are specified as:
\begin{itemize}
	\item \emph{\textbf{Uniform flow:}} the potential associated with a uniform flow is given by:
	\begin{equation}
		\Phi^{(u)}(x,y)
		= U_\infty\,x + V_\infty\,y.
	\end{equation}
	from which we obtain the velocity field as:
	\begin{equation}
		\mathbf W^{(u)}(x,y)
		= \nabla \Phi^{(u)}
		= \bigl(U_\infty,\,V_\infty\bigr).
	\end{equation}
\item \emph{\textbf{Regularized point vortices:}} the stream function for each vortex of circulation $\Gamma_k$ at $(x_{0,k},\,y_{0,k})$ with core radius $R_{0,k}^{(v)}$ is:
\begin{equation}
	\Psi^{(v)}_k(x,y)
	= \frac{\Gamma_k}{4\pi}
	\ln\Big[r_k^2+ \bigl(R_{0,k}^{(v)}\bigr)^2\Big].
\end{equation}
from which the velocity field is obtained by:
\begin{equation}
\begin{split}
	&\mathbf W^{(v)}_k(x,y)
	= \nabla^\perp \Psi^{(v)}_k
	=\begin{pmatrix}
		-\partial_y \Psi^{(v)}_k\\[1ex]
		\phantom{-}\partial_x \Psi^{(v)}_k
	\end{pmatrix}
	= \frac{\Gamma_k}{2\pi}\,
	\frac{1}{r_k^2 + \bigl(R_{0,k}^{(v)}\bigr)^2}
	\begin{pmatrix}
		-\, (y - y_{0,k})\\[1ex]
		\;\; (x - x_{0,k})
	\end{pmatrix}.
\end{split}
\end{equation}
where $r_k^2 = (x - x_{v,k})^2 + (y - y_{v,k})^2$.
%
	\item \emph{\textbf{Regularized dipoles:}} each dipole of vector moment $\bm\mu_\ell=(\mu_{x,\ell},\,\mu_{y,\ell})$ at $(x_{d,\ell},\,y_{d,\ell})$ with regularization radius $R_{0,\ell}^{(d)}$ has potential:
	\begin{equation}
		\Phi_\ell(x,y)
		= \frac{\bm\mu_\ell\!\cdot\!\bigl(x - x_{d,\ell},\,y - y_{d,\ell}\bigr)}
		{2\pi\,\bigl[r_\ell^2 + \bigl(R_{0,\ell}^{(d)}\bigr)^2\bigr]}\,.
	\end{equation}
	from which we obtain the velocity field as:
	\begin{equation}
		\begin{split}
			&\mathbf W^{(d)}_\ell(x,y)
			= \nabla \Phi_\ell
			=\\
			 &\frac{1}{2\pi} \,
			\frac{1}{\bigl(r_\ell^2 + (R_{0,\ell}^{(d)})^2\bigr)^2}
			\begin{pmatrix}
				\mu_{x,\ell}\bigl[(x - x_{d,\ell})^2 - (y - y_{d,\ell})^2\bigr]
				+ 2\,\mu_{y,\ell}\,(x - x_{d,\ell})(y - y_{d,\ell})\\[1ex]
				\mu_{y,\ell}\bigl[(y - y_{d,\ell})^2 - (x - x_{d,\ell})^2\bigr]
				+ 2\,\mu_{x,\ell}\,(x - x_{d,\ell})(y - y_{d,\ell})
			\end{pmatrix}.
		\end{split}
	\end{equation}
	where $r_\ell^2 = (x - x_{d,\ell})^2 + (y - y_{d,\ell})^2$.	
	\item \emph{\textbf{Regularized point sources/sinks:}} each source (positive strength) or sink (negative strength) of net strength $Q_j$ at $(x_{s,j},\,y_{s,j})$ with core radius $R_{0,j}^{(s)}$ has potential:
	\begin{equation}
		\Phi^{(s)}_j(x,y)
		= \frac{Q_j}{4\pi}\,\ln\!\Bigl[r_j^2 + \bigl(R_{0,j}^{(s)}\bigr)^2\Bigr]\,.
	\end{equation}
	from which we obtain:
	\begin{equation}
		\mathbf W^{(s)}_j(x,y)
		= \nabla \Phi^{(s)}_j
		= \frac{Q_j}{2\pi}\,
		\frac{1}{r_j^2 + \bigl(R_{0,j}^{(s)}\bigr)^2}
		\begin{pmatrix}
			x - x_{s,j}\\[1ex]
			y - y_{s,j}
		\end{pmatrix}.
	\end{equation}
\end{itemize}
where $r_j^2 = (x - x_{s,j})^2 + (y - y_{s,j})^2$.

In each case, the regularization radius $R_{0,k}^{(v)}$, $R_{0,\ell}^{(d)}$, and $R_{0,j}^{(s)}$ (for vortices, dipoles, and sources/sinks respectively) prevents singular behavior at the core.  By construction, every component satisfies: 
\begin{equation}
	\nabla\!\cdot\,\mathbf W^{(u)} = 0,\quad
	\nabla\!\cdot\,\mathbf W^{(v)}_k = 0,\quad
	\nabla\!\cdot\,\mathbf W^{(d)}_\ell = 0,\quad \nabla\!\cdot\,\mathbf W^{(s)}_j = 0.
\end{equation}
Therefore, by linearity of divergence:
\begin{equation}
	\nabla\!\cdot\mathbf W(x,y)
	=\nabla\!\cdot\bigl(\mathbf W^{(u)}+\sum_k\mathbf W^{(v)}_k+\sum_\ell\mathbf W^{(d)}_\ell+\sum_j\mathbf W^{(s)}_j\bigr)
	=0.
\end{equation}
The proposed composite inviscid wind model is characterized by the following free parameters, which can be fine-tuned using the available wind data:
\begin{equation}
	\textbf{Parameter Count}=\underbrace{U_\infty,\,V_\infty}_{2}
	\;+\;
	\underbrace{4M_v}_{\Gamma_k,\,(x_{v,k},y_{v,k}),\,R^{(v)}_{0,k}}
	\;+\;
	\underbrace{5M_d}_{(\mu_{x,\ell},\,\mu_{y,\ell}),\,(x_{d,\ell},y_{d,\ell}),\, R^{(d)}_{0,\ell}}+\;
	\underbrace{4M_s}_{Q_j,\,(x_{s,j},y_{s,j}),\, R^{(s)}_{0,j}}.
\end{equation}

\subsection{The Model for Flight-Sensitive Areas}\label{subsec22}
The function \( g(x,y) \) in Eq. (\ref{eq1cc}) characterizes flight-sensitive areas. While \( g(x,y) \) can, in principle, approximate a full sensitivity map corresponding to various hazardous phenomena, in this work it is used to represent  high-probability hazardous zones for generality. Accordingly, \( g(x,y) \) typically attains its maximum at the geometric center of each flight-sensitive region and decays with spatial distance.

In this study, flight-sensitive areas are modeled as ellipses centered at the respective hazardous zones. The function \( g(x,y) \) is constructed to act as a soft penalty term with tunable intensity:
\footnote{Alternatively, hard constraints can be embedded directly into the Lagrangian, for example through \( g_i(x,y) := g_{h,i}(x,y) = \exp\!\left[k_i\left(1 - \|X_s - X_{sc,i}\|_{\mathbf{A}}\right) + d_i\right] \), where \( k_i = -d_i + c_i \), and \(\exp(d_i)\) and \(\exp(c_i)\) represent the penalty values at the ellipse perimeter and center, respectively. The elliptical geometry, defined by \(\|X_s - X_{sc,i}\|_{\mathbf{A}} \leq 1\), \( i = 1,\dots,n \), can equivalently be enforced as a set of inequality constraints, i.e., \(\mathcal{S}_{s,i}(X_s) \triangleq 1 - \|X_s - X_{sc,i}\|_{\mathbf{A}} \leq 0\), for \(i=1,\dots,n\), \(\forall t \in [0,t_f]\), where \(\mathcal{S}_{s,i}\) denotes the state inequality constraint associated with the \(i^{\text{th}}\) flight-sensitive area.}
\begin{equation}\label{eq6}
	\begin{split}
    &g(x,y):=\sum_{i=1}^{n}\mathbf{c}_{s,i}g_i(x,y),\quad g_i(x,y)\triangleq
g_{s,i}(x,y).
    \end{split}
\end{equation}
where:
\begin{equation}\label{eq6a}
	\begin{split}
		& g_{s,i}(x,y)=\frac{1}{{||X_s-X_{sc,i}||}_{\textbf{A}}},\quad X_s:=	
		\begin{pmatrix}
			x\\
			y		
		\end{pmatrix},\quad ||\bullet||_{\textbf{A}}:=\sqrt{(\bullet)^{\textbf{T}}\textbf{A}\,(\bullet)},\\
	 &\textbf{A}=\textbf{R}\textbf{D}
		{\textbf{R}}^{\textbf{T}},\quad \textbf{D}=\begin{pmatrix}
			\frac{1}{a_i^2}&0\\
			0&\frac{1}{b_i^2}		
		\end{pmatrix},\quad \textbf{R}=
		\begin{pmatrix}
			\cos(\alpha_i)&-\sin(\alpha_i)\\
			\sin(\alpha_i)&\cos(\alpha_i)		
		\end{pmatrix}.
	\end{split}
\end{equation}
In the above, \( g_s \) denotes the aggregated penalty term; \( c_{s,i} \) is the weight associated with the \( i^{\text{th}} \) flight-sensitive area; \( a_i \) and \( b_i \) are the semi-major and semi-minor axes, respectively; \( X_{sc,i} \) is the geometrical center; and \( \alpha_i \) denotes the orientation of the corresponding elliptical region.


\section{Turnpike Property Applied to Cruise Dynamics}\label{sec3}
%
The turnpike property describes a phenomenon in many OCPs where, for different initial conditions and varying time horizons, the solutions tend to converge to a neighborhood of a particular steady state solution and remain there for most of the time horizon \cite{turnpike1,turnpike2}. 

\begin{customassumption}[Monotonicity of the state]
	Let $(x,y,m,v)\colon [0,t_f]\to\mathbb{R}^4$ be of class \(\mathcal C^1\).  We assume:
	\begin{equation}
		\frac{dx}{dt}(t)>0,\quad \frac{dy}{dt}(t)>0,\quad \frac{dm}{dt}(t)<0,
		\qquad
		\forall\,t\in[0,t_f].
	\end{equation}
	Hence \(x\) and \(y\) are strictly increasing, \(m\) is strictly decreasing, and the only state–variable whose time–derivative may change sign is the speed \(v\).  It follows that \(v\) alone can exhibit a nontrivial turnpike behavior.
\end{customassumption}


To analyze the turnpike property with respect to \( v \), we first extract the stationary point by solving:
\begin{equation}\label{eq9}
	\begin{split}
		&\frac{dv}{dt}=0,\ \Rightarrow \Pi(t)=\frac{D(m,v)}{T_{max}(v)}.
	\end{split}
\end{equation} 
The stability of the stationary point can be analyzed using linear perturbation theory. However, since the aircraft speed and mass dynamics are coupled, the perturbation analysis must theoretically account for both variables, making the problem more complex. To simplify this, we assume that at the stationary point, the aircraft mass \( m \) remains constant, thereby reducing the analysis to the perturbed field associated only with the aircraft speed \( v \).
\begin{customassumption}\label{assumption0}
	Let \(m>0\) be fixed, and \(\Pi^*\in[0,1]\). Define:
	\begin{equation}
		F_v(v) \;=\; \frac{\Pi^*\,T_{\max}(v) - D(m,v)}{m}.
	\end{equation}
	such that \(F_v\) is of class \(\mathcal C^2\) in a neighbourhood of $v^*$ satisfying  \(\frac{\partial F_v}{\partial v}(v^*)<0\), where \(v^*>0\) is the unique solution of:
	\begin{equation}
		F_v(v^*) = 0
		\quad\Longleftrightarrow\quad
		\Pi^*\,T_{\max}(v^*) = D(m,v^*).
	\end{equation}
\end{customassumption}
\begin{theorem}[Local exponential turnpike for the aircraft speed]
	
	 Under Assumption \ref{assumption0}, there exists \(\delta>0\) and \(\mu>0\)
	such that for any solution \(v(t)\) of \(\frac{dv}{dt}=F_v(v)\) with
	\(|v(0)-v^*|\le\delta\), one has:
	\begin{equation}
		|v(t)-v^*|\;\le\;|v(0)-v^*|\,e^{-\mu t},
		\quad\forall t\ge0.
	\end{equation}
	In particular, \(v^*\) is locally exponentially asymptotically stable, implying that the optimal aircraft
	speed exhibits the turnpike property, i.e., for any \(\varepsilon>0\) there
	exists \(T_0>0\) such that for all sufficiently large horizon
	\(t_f\), the optimal trajectory satisfies
	\(|v(t)-v^*|\le\varepsilon\) for all
	\(t\in [T_0,\;t_f-T_0]\).
\end{theorem}

\begin{proof}
	Under Assumption \ref{assumption0}, \(F_v\in\mathcal C^2\) and \(\frac{\partial F_v}{\partial v}(v^*)<0\).  We set:
	\begin{equation}
		\lambda:=\frac{\partial F_v}{\partial v}(v^*)\;<0,\qquad \varepsilon(t):=v(t)-v^*.
	\end{equation}
	By Taylor’s theorem, for \(\lvert \varepsilon\rvert\) small:
	\begin{equation}
			F_v(v^*+\varepsilon)
		=\underbrace{F_v(v^*)}_{0}+\lambda\,\varepsilon+r(\varepsilon).
	\end{equation}
	where \(r(\varepsilon)=\tfrac12\,\frac{\partial^2 F_v}{\partial v^2}(\xi)\,\varepsilon^2\) for some \(\xi\) between \(v^*\) and \(v^*+\varepsilon\).  Since \(\frac{\partial F_v}{\partial v}\) is continuous at \(v^*\), there exists \(\delta>0\) such that:
	\begin{equation}
		\bigl\lvert r(\varepsilon)\bigr\rvert
		=\bigl\lvert F_v(v^*+\varepsilon)-F_v(v^*)-\lambda\,\varepsilon\bigr\rvert
		\;\le\;\frac{\lvert\lambda\rvert}{2}\,\lvert\varepsilon\rvert
		\quad\text{whenever }|\varepsilon|\le\delta.
	\end{equation}
	Hence any solution of \(\frac{dv}{dt}=F_v(v)\) with \(v(0)\in(v^*-\delta,v^*+\delta)\) satisfies:
	\begin{equation}
		\frac{d \varepsilon}{dt}
		=\lambda\,\varepsilon+r(\varepsilon),
		\quad
		\bigl\lvert r(\varepsilon)\bigr\rvert\le\frac{|\lambda|}{2}\lvert\varepsilon\rvert.
	\end{equation} 
	If \(\lvert\varepsilon(t)\rvert\le\delta\), then:
	\begin{equation}
			\frac{d}{dt}\lvert\varepsilon\rvert
		\;=\;\mathrm{sgn}(\varepsilon)\,\frac{d \varepsilon}{dt}
		\;\le\;\lvert\lambda\rvert\,\lvert\varepsilon\rvert-\bigl|\lambda\bigr|\tfrac12\,\lvert\varepsilon\rvert
		\;=\;\frac{\lambda}{2}\,\lvert\varepsilon\rvert.
	\end{equation}
	since \(\lambda<0\).  By Grönwall’s inequality:
	\begin{equation}
		\lvert\varepsilon(t)\rvert
		\;\le\;\lvert\varepsilon(0)\rvert\,e^{(\lambda/2)\,t}, 
		\qquad t\ge0.
	\end{equation}
	and \(\lvert\varepsilon(t)\rvert<\delta\) for all \(t\).  Therefore:
	\begin{equation}
		\bigl\lvert v(t)-v^*\bigr\rvert
		=\lvert\varepsilon(t)\rvert
		\;\le\;\lvert v(0)-v^*\rvert\,e^{(\lambda/2)\,t}.
	\end{equation}
	showing \(v^*\) is (locally) exponentially asymptotically stable.
\end{proof}
To validate Assumption \ref{assumption0}, we numerically evaluated 
\(\lambda := \frac{\partial F_v}{\partial v}(v^*)\) over a 
\(\Pi^*\!-\!m\) domain. The computations consistently gave
\(\frac{\partial F_v}{\partial v}(v^*) < 0\) across the entire domain, thereby confirming Assumption \ref{assumption0}.



\section{The Surrogate Optimization Framework}\label{sec4}
\begin{customassumption}[Quasi‐steady flight]\label{assumtion00}
	Let $0 = t_{0} < t_{1} < \cdots < t_{N} = t_{f}$
	be a partition of the time‐horizon with mesh size 
	\(\displaystyle \Delta t := \max_{0\le n<N}(t_{n+1}-t_n)\).  We assume the aircraft speed $v:[0,t_f]\to\mathbb{R}$
	is piecewise constant on this partition, i.e.:
\begin{equation}
	\Omega_n := [t_n, t_{n+1}), 
	\quad
	v(t) = v_n,\;t\in\Omega_n,\;n=0,\dots,N-1
	\;\;\Longrightarrow\;\;
	\frac{d v}{dt}(t) = 0,\;t\in (t_n,t_{n+1})\,.
\end{equation}
Moreover, in the continuous limit $\Delta t\to0$, the admissible speed control $v(\cdot)$ becomes any Lebesgue–measurable function $v\in L^\infty([0,t_f])$.
\end{customassumption}

%

Under Assumption \ref{assumtion00}, the surrogate OCP becomes:
\begin{equation}\label{eq1c}
	\begin{split}
		&\min_{v(t),\chi(t),t_f}\mathcal{J}:=\mathbf{c}_tt_f+\mathbf{c}_mm_f+\int_{0}^{t_f}\sum_{i=1}^{n}\mathbf{c}_{s,i}g_i(x,y) dt.
	\end{split}
\end{equation}

\ \ \ \textbf{s.t.},
\begin{equation}\label{eq1}
	\begin{split}
		&\textbf{1) State Dynamics:}\\
		&\frac{dx}{dt}=v(t)\cos(\chi(t))+W_x(x,y),\\
		&\frac{dy}{dt}=v(t)\sin(\chi(t))+W_y(x,y),\\
		&\frac{dm}{dt}=-D(m,v(t),\bar{h}) C_s(v(t),\bar{h}),\\
		&\textbf{2) Boundary Conditions:}\\
		&x(0)=x_0,\quad y(0)=y_0,\quad m(0)=m_0,\quad x(t_f)=x_f,\quad y(t_f)=y_f,\\
		&\textbf{3) Control Bounds:}\\
		&\chi_{min}\leq\chi({t})\leq\chi_{max},\quad M_{min}\big(c_0\bar{R}\Theta(\bar{h})\big)^{\frac{1}{2}}\leq {v}({t})\leq M_{max}\big(c_0\bar{R}\Theta(\bar{h})\big)^{\frac{1}{2}},\quad \forall t\in[0,t_f],\\
		&\textbf{4) Mixed-Inequality Constraints:}\\
		&\Pi_{min}\leq\Pi(t)\leq\Pi_{max},\quad \Pi:=\frac{D(m,v,\bar{h})}{T_{max}(v,\bar{h})},\quad \forall t\in[0,t_f].
	\end{split}
\end{equation}

\subsection{Compact Form Notation}\label{subsec41}  
We translate the above control problem into the following Mayer form:
\begin{equation}\label{eq32}
	\begin{split}
		&\min_{U(t),t_f}\mathcal{J}:=\Phi(X_f,t_f)\\
		&\textbf{s.t.},\\
		&\frac{dX}{dt}=F\bigl(X,U(t)\bigr),\\
		&\mathcal{C}\bigl(U(t)\bigr)\leq 0,\quad \mathcal{S}\bigl(X,U(t)\bigr)\leq 0,\quad \forall t\in[0,t_f],\\
		&\phi^T_0(X_0):=\bigl(x_0-x(0),y_0-y(0),m_0-m(0),z_0\bigr)=\vec{0},\\
		&\phi^T_f(X_f):=\bigl(x_f-x(t_f),y_f-y(t_f),0,0\bigr)=\vec{0}.
	\end{split}
\end{equation}

In the above, \(X:[0,t_f]\rightarrow\mathbb{R}^4\) represents the state vector defined as $X^T:=\bigl(x,y,m,z\bigr)$; \(X_0\) and \(X_f\) denote the initial and final conditions, respectively; \(U:[0,t_f]\rightarrow\mathbb{R}^2\) denotes the control vector defined as $U^T:=\bigl(v(t),\chi(t)\bigr)$; \(F\bigl(X,U(t)\bigr):\mathbb{R}^4\times \mathbb{R}^2\rightarrow\mathbb{R}^4\) represents the dynamic vector defined as $F^T:=\bigl(F_x,F_y,F_m,F_z\bigr)$,
where $\frac{dz}{dt}=:F_z=\sum_{i=1}^{n}\mathbf{c}_{s,i}\,g_{s,i}(x,y)$;
and \(\Phi: \mathbb{R}^4\times \mathbb{R}\rightarrow \mathbb{R}\) is the objective function in Mayer format defined as $\Phi=\mathbf{c}_{t}\,t_f+\mathbf{c}_{m}\,m_f+z_f,$
where $z_f=\int_{0}^{t_f}\sum_{i=1}^{n}\mathbf{c}_{s,i}\,g_{s,i}(x,y)\,dt$. Moreover, \(\mathcal{C}\bigl(U\bigr): \mathbb{R}^2\rightarrow\mathbb{R}^{4}\) and \(\mathcal{S}\bigl(X,U(t)\bigr):\mathbb{R}^4\times\mathbb{R}^2\rightarrow\mathbb{R}^2\) are column vectors representing inequality constraints.

Under the assumption of normality (homogeneity) and using the direct-adjoining method (see \cite{26}, \cite{27}), we define the pseudo-Hamiltonian as:
\begin{equation}\label{eq34}
	\begin{split}
		\mathcal{H}(X,U(t),\lambda,\eta):=\langle\lambda,F\bigl(X,U(t)\bigr)\rangle+\langle\eta,\mathcal{S}(X,U(t))\rangle.
	\end{split}
\end{equation}
Here, \(\lambda:[0,t_f]\rightarrow\mathbb{R}^4\) represents the co-state vector defined as $\lambda^T:=\bigl(\lambda_{x},\lambda_{y},\lambda_{m},\lambda_{z}\bigr)$, and \(\eta:[0,t_f]\rightarrow\mathbb{R}^2\) stands for time-dependent multipliers associated with the state-inequality constraints.     

The minimization conditions \(\forall t\in[0,t_f]\) read:
\begin{equation}\label{eq35}
	\begin{split}
		\mathcal{H}(X^*,U^*(t),\lambda^*,\eta^*)&=\min_{\mathcal{C}(U)\leq0}\mathcal{H}(X^*,U(t),\lambda^*,\eta^*),\\[1mm]
		\langle\eta^*,\mathcal{S}(X,U^*(t))\rangle&=0,\quad \eta^*\geq0.
	\end{split}
\end{equation}

The co-state dynamics and the transversality conditions are:
\begin{equation}\label{eq36}
	\begin{split}
		\frac{d{\lambda}^T}{dt}&=-\frac{\partial\mathcal{H}}{\partial {X}},\quad \lambda_0:=\lambda(0)=-\Bigl(\frac{\partial{\phi}_0}{\partial {X}_0}\Bigr)^T\nu_0,\quad \lambda_f:=\lambda(t_f)=\frac{\partial\Phi}{\partial X_f}+\Bigl(\frac{\partial{\phi}_f}{\partial {X}_f}\Bigr)^T\nu_f,\\[2mm]
		\mathcal{H}_{f}&:=\mathcal{H}(t_f)=-\frac{\partial\Phi}{\partial t_f}-\nu_f^T\frac{\partial{\phi}_f}{\partial t_f}.
	\end{split}
\end{equation}

Let \(\tau\) be a possible time instant within the mixed-boundary arc at which the co-state variables are discontinuous. Therefore, from \cite{27}, we have:
\begin{equation}\label{eq37}
	\begin{split}
		\lambda^T(\tau^{-})&=\lambda^T(\tau^{+})-\nu^T(\tau)\frac{\partial \mathcal{S}}{\partial X}\Big|_{t=\tau},\quad \nu(\tau)\geq0,\quad \langle\nu(\tau),\mathcal{S}\bigl(X,U(\tau)\bigr)\rangle=0,\\[1mm]
		\mathcal{H}(\tau^-)&=\mathcal{H}(\tau^+).
	\end{split}
\end{equation}

\subsection{The Optimal Heading Angle, \(\chi(t)\)}
From the Hamiltonian in Eq. (\ref{eq35}), the optimal heading angle for the interior arc is obtained through: 
\begin{equation}\label{eq411}
	\begin{split}
		\frac{\partial \mathcal{H}}{\partial \chi}&=0\Rightarrow \tan\bigl(\chi(t)\bigr)=\frac{\lambda_{y}}{\lambda_{x}}\Rightarrow \frac{d\chi}{dt}\Bigl(1+\tan^2(\chi)\Bigr)=\frac{\frac{d\lambda_{y}}{dt}\lambda_{x}-\frac{d\lambda_{x}}{dt}\lambda_{y}}{\lambda_{x}^2}.
	\end{split}
\end{equation}  

From the co-state dynamics, we have: 
\begin{equation}\label{eq42}
	\begin{split}
		\frac{d\lambda_{x}}{dt}&=-\frac{\partial \mathcal{H}}{\partial x}=-\lambda_{z}\frac{\partial }{\partial x}\sum_{i=1}^{n}\mathbf{c}_{s,i}\,g_{s,i}(x,y)-\Bigl(\lambda_{x}\frac{\partial W_x}{\partial x}+\lambda_{y}\frac{\partial W_y}{\partial x}\Bigr),\\[2mm]
		\frac{d\lambda_{y}}{dt}&=-\frac{\partial \mathcal{H}}{\partial y}=-\lambda_{z}\frac{\partial }{\partial y}\sum_{i=1}^{n}\mathbf{c}_{s,i}\,g_{s,i}(x,y)-\Bigl(\lambda_{x}\frac{\partial W_x}{\partial y}+\lambda_{y}\frac{\partial W_y}{\partial y}\Bigr),\\[2mm]
		\frac{d\lambda_{z}}{dt}&=-\frac{\partial \mathcal{H}}{\partial z}=0,\Rightarrow \lambda_{z}(t)=\lambda_{z}(t_f)=\frac{\partial \Phi}{\partial z_f}=1,\quad \forall t\in[0,t_f].
	\end{split}
\end{equation} 

Notably, the co-state dynamics described above do not incorporate the multipliers \(\eta\), as the state-inequality constraints (i.e., Mach-inequality constraints) depend exclusively on the aircraft speed, \(v\).

Plugging Eq. (\ref{eq42}) into Eq. (\ref{eq411}), we obtain:
\begin{equation}\label{eq43}
	\begin{split}
		&\frac{dq}{dt}=\Biggl(-\frac{\partial W_x}{\partial y}+\Bigl(\frac{\partial W_x}{\partial x}-\frac{\partial W_y}{\partial y}\Bigr)q+\Bigl(\frac{\partial W_y}{\partial x}\Bigr)q^2\Biggr)+\\
		&\frac{1}{\lambda_{x}}\Biggl(q\frac{\partial}{\partial x}\sum_{i=1}^{n}\mathbf{c}_{s,i}\,g_{s,i}(x,y)-\frac{\partial}{\partial y}\sum_{i=1}^{n}\mathbf{c}_{s,i}\,g_{s,i}(x,y)\Biggr),\quad q:=\tan\bigl(\chi(t)\bigr).    
	\end{split}
\end{equation}
We observe that the optimal heading angle remains the same as that of the steady-state problem. Specifically, Eq. (\ref{eq43}) delineates the optimal \(\chi(t)\) for the interior arc, even with active state-inequality constraints. Furthermore, together with the assumption of continuity, any (possible) boundary arc for \(\chi(t)\) can be computed straightforwardly.

\subsection{The Optimal Aircraft Speed ${v}({t})$}\label{subsec44}
Let us initially suppose the mixed-inequality constraints are inactive, i.e., $\eta=0$, $\forall {t} \in [0, {t}_f]$. Since the Hamiltonian is autonomous, we write: 
\begin{equation}\label{eq41b}
	\begin{split}
		&\mathcal{H}(t):=\lambda_xF_x+\lambda_yF_y+\lambda_mF_m+F_z=\mathcal{H}(t_f)=-\mathbf{c}_t,\quad \forall t\in [0,t_f].
	\end{split}
\end{equation} 
Therefore, the optimal aircraft speed can be obtained as:
\begin{equation}\label{eq41a}
	\begin{split}
		&\frac{\partial \mathcal{H}}{\partial {v}}=0\Rightarrow \lambda_x\frac{\partial F_x}{\partial \tilde{v}}+\lambda_y\frac{\partial F_y}{\partial \tilde{v}}+\lambda_m\frac{\partial F_m}{\partial \tilde{v}}+\cancelto{0}{\frac{\partial F_z}{\partial \tilde{v}}}=0.
	\end{split}
\end{equation} 

On using $\lambda_y=\tan\big(\chi(\tilde{t})\big)\lambda_x$ in the above equation, we can write:
\begin{equation}\label{eq41c}
	\begin{split}
		&\lambda_m=\frac{-\mathbf{c}_t-F_z-\lambda_xF_x-\lambda_xF_y\tan\big(\chi(\tilde{t})\big)}{F_m}.
	\end{split}
\end{equation} 
Therefore, Eq. (\ref{eq41a}) becomes:
\begin{equation}\label{eq41d}
	\begin{split}
		&\lambda_x\frac{\partial F_x}{\partial {v}}+\lambda_x\tan\big(\chi({t})\big)\frac{\partial F_y}{\partial {v}}-\bigg(\mathbf{c}_t+F_z+\lambda_xF_x+\lambda_xF_y\tan\big(\chi({t})\big)\bigg)\frac{1}{F_m}\frac{\partial F_m}{\partial {v}}=0.
	\end{split}
\end{equation}
where:
\begin{equation}\label{eq41}
	\begin{split}
		&\frac{\partial F_m}{\partial {v}}=-\left(\frac{\partial C_s}{\partial{v}}D+\frac{\partial D}{\partial{v}}C_s\right),\quad \frac{\partial F_x}{\partial {v}}=\cos(\chi({t})),\quad \frac{\partial F_y}{\partial {v}}=\sin(\chi({t})).
	\end{split}
\end{equation}

\subsubsection{Special Cases}
\subsubsection*{A) Minimum Fuel Problem:} 
Let us consider a scenario where $\mathbf{c}_t = F_z = 0$. In this case, Eq. (\ref{eq41d}) becomes:
\begin{equation}\label{eq42d}
	\begin{split}
		&\frac{\partial \mathcal{H}}{\partial {v}}(t)=\lambda_x\frac{\partial F_x}{\partial {v}}+\lambda_x\tan\big(\chi({t})\big)\frac{\partial F_y}{\partial {v}}-\bigg(\lambda_xF_x+\lambda_xF_y\tan\big(\chi({t})\big)\bigg)\frac{1}{F_m}\frac{\partial F_m}{\partial {v}}=0.
	\end{split}
\end{equation} 
The above equation can be simplified as:
\begin{equation}\label{eq43d}
	\begin{split}
		&\frac{\partial \mathcal{H}}{\partial {v}}(t)=\frac{\lambda_x}{\cos\big(\chi(t)\big)}-\bigg(\frac{v\lambda_x}{\cos\big(\chi(t)\big)}+\lambda_{x}\big(W_x+\tan\big(\chi(t)\big)W_y\big)\bigg)\frac{F_{m,v}}{F_m}=\\
		&\frac{\lambda_x}{\cos\big(\chi(t)\big)}-P(t)\frac{F_{m,v}}{F_m}=0.
	\end{split}
\end{equation}
where $P(t):=\frac{\lambda_x}{\cos\big(\chi(t)\big)}\bigg(v+\big(W_x\cos\big(\chi(t)\big)+\sin\big(\chi(t)\big)W_y\big)\bigg),\ \forall t\in[0,t_f]$.\\
\textbf{Definitions:} $F_{m,v}:=\frac{\partial F_m}{\partial {v}}$, $F_{m,vv}:=\frac{\partial^2F_m}{\partial v^2}$, $F_{m,m}:=\frac{\partial F_m}{\partial m}$.
\begin{customassumption}[Monotonicity of Fuel Flow and Drag]\label{assumption1}
	Let \( C_s (v) : \mathbb{R} \to \mathbb{R}_{>0} \) and \( D(m,v) : \mathbb{R} \times \mathbb{R} \to \mathbb{R}_{>0} \) be such that \( C_s(v) \) is non-decreasing in \( v \), and \( D(m,v) \) is non-decreasing in both \( v \) and \( m \).
\end{customassumption}

\begin{customassumption}[]\label{assumption2}
	$F_{m,vv}<0$. 
\end{customassumption}
\begin{customassumption}[Boundedness of the Optimal Speed]\label{assumption4}
	For each \(t\in[0,t_f]\) the first‑order optimality condition $\frac{\partial \mathcal{H}}{\partial v}(t)\;=\;0$ can be written in the form:
\begin{equation}\label{eq49}
	G\bigl(m(t),\,v^*(t),\,W_x,\,W_y,\,\bar h\bigr)\;=\;0.
\end{equation}
	and that this equation implicitly defines a unique mapping:
	\begin{equation}\label{eq50a}
		v^*(t) \;=\; f\bigl(m(t),\,W_x,\,W_y,\,\bar h\bigr).
	\end{equation}
	Then, the resulting control is bounded in \((v_{\min},\,v_{\max})\), i.e. $v_{\min} \;<\; v^*(t) \;<\; v_{\max},\ \forall\,t\in[0,t_f]$. 
\end{customassumption}

\begin{theorem}[Strict Convexity of the Optimal Speed $v^*(t)$]\label{thm}
Under Assumptions \ref{assumption1}, \ref{assumption2}, and \ref{assumption4} the minimizer $v^*(t)$ of the Hamiltonian for the minimum fuel problem:
\begin{equation}
	\min_{v(t)\in[v_{min},v_{max}],\chi(t),t_f}\mathcal{J}:=-m_f,
\end{equation}
satisfies $v_{min} < v^*(t) < v_{max}, \forall\,t\in[0,t_f]$, and is strictly convex. 
\end{theorem}
\begin{proof}
	We show that $\frac{\partial \mathcal{H}}{\partial v}(t) = 0 \;\Longrightarrow\; \frac{\partial^2 \mathcal{H}}{\partial v^2}(t) > 0,\ \forall t \in [0,t_f]$. From Eq. (\ref{eq42d}), we obtain:
	\begin{equation}\label{eq50}
	\frac{\partial^2 \mathcal{H}}{\partial {v^2}}(t)=-\bigg(\frac{\lambda_{x}}{\cos\big(\chi(t)\big)}\bigg)\frac{F_{m,v}}{F_m}-P(t)\frac{F_{m,vv}F_m-F_{m,v}^2}{F_m^2}.
    \end{equation}
 Plugging Eq. (\ref{eq43d}) into Eq. (\ref{eq50}), we obtain:
	\begin{equation}\label{eq51}
	\frac{\partial^2 \mathcal{H}}{\partial {v^2}}(t)=-P(t)\frac{F_{m,vv}}{F_m}.
   \end{equation}
On the other hand, from the Hamiltonian, we have:
\begin{equation}\label{eq52}
	\begin{split}
		&\mathcal{H}(t):=\lambda_xF_x+\lambda_x\tan\big(\chi(t)\big)F_y+\lambda_mF_m=P(t)+\lambda_mF_m=0,\quad \forall t\in [0,t_f].
	\end{split}
\end{equation} 
From:
\begin{equation}\label{eq54}
	\frac{d\lambda_m}{dt}
	= -\frac{\partial\mathcal H}{\partial m}
	= -\,\lambda_m\,F_{m,m},\quad \lambda_{m}(t_f)=\frac{\partial [-m_f]}{\partial m_f}=-1.
\end{equation}
and since \(F_m=-C_s(v)\,D(m,v)<0\), we have: $-F_{m,m}
= C_s(v)\,\frac{\partial D}{\partial m}
=:a(t)$, where by Assumption \ref{assumption1}, \(a(t)>0\) for all \(t\). Thus: $\frac{d\lambda_m}{dt} = a(t)\,\lambda_m,\
\lambda_m(t_f)=-1.$
The general solution of Eq. (\ref{eq54}) is:
\begin{equation}
	\lambda_m(t)
	= \lambda_m(t_f)\,\exp\!\Bigl(-\!\int_{t}^{t_f}a(s)\,ds\Bigr)
	= -\exp\!\Bigl(-\!\int_{t}^{t_f}a(s)\,ds\Bigr).
\end{equation}
Since $\lambda_m(t)<0,\ \forall\,t\in[0,t_f]$, we have: $\operatorname{sgn}\bigl(\lambda_mF_m\bigr)
=+1$, and thus, from Eq. (\ref{eq52}), $\operatorname{sgn}\big(P(t)\big)<0$.

From \eqref{eq51}, since \(F_m<0\) and we showed \(P<0\), it follows that: $\operatorname{sgn}\!\Bigl(\frac{\partial^2 \mathcal{H}}{\partial v^2}(t)\Bigr)
=\operatorname{sgn}\bigl(F_{m,vv}\bigr)$, where by Assumption \ref{assumption2}:
\begin{equation}
	\frac{\partial^2 \mathcal{H}}{\partial v^2}(t)>0,
	\quad \forall\,t\text{ such that }\frac{\partial \mathcal{H}}{\partial v}(t)=0.
\end{equation}
Therefore the extremum is a local minimum of the Hamiltonian, completing the proof.
\end{proof}
\begin{theorem}\label{theo4}
	Let $0\le t\le t_f$ and suppose: $\frac{d\lambda_m}{dt}(t)=a(t)\,\lambda_m(t),\
	\lambda_m(t_f)=-1$
	on the Banach space $\mathbb{R}$ with norm $\|\cdot\|=|\cdot|$, where 
	$a\in L^\infty([0,t_f])$ and $0 < a_{\min} \le a(t)\le a_{\max}$. Then the unique mild (and classical) solution $\lambda_m(\cdot)$ satisfies, for every $t\in[0,t_f]$:
	\begin{equation}
		\lambda_m(t)
		\;=\;
		-\exp\!\Bigl(-\!\!\int_{t}^{t_f}a(s)\,ds\Bigr)
	\end{equation}
	and in particular:
	\begin{equation}
		-1\le\;
		\lambda_m(t)
		\;\le\;
		-\,\exp\bigl(-a_{\max}t_f\bigr).
	\end{equation}
\end{theorem}
\begin{proof}
	The solution of $\lambda_{m}(t)$ reads $\lambda_m(t)
	=\exp\!\Bigl(\!\int_{t_f}^{t}a(s)\,ds\Bigr)\,\lambda_m(t_f)
	=-\exp\!\Bigl(-\!\int_{t}^{t_f}a(s)\,ds\Bigr)$. Since \(a_{\min}(t_f-t)\le\int_t^{t_f}a(s)\,ds\le a_{\max}(t_f-t)\),
	exponentiation and multiplication by \(-1\) yield:
	\begin{equation}\label{eq62a}
	-\,\exp\!\bigl(-a_{\min}(t_f - t)\bigr)
	\;\le\;
	\lambda_m(t)
	\;\le\;
	-\,\exp\!\bigl(-a_{\max}(t_f - t)\bigr).
    \end{equation}
	Under the hypotheses of the Theorem, since $		\frac{d\lambda_m}{dt}(t)
	= a(t)\,\lambda_m(t),\ a(t)>0,\ \lambda_m(t)<0$,
	we get $\frac{d\lambda_m}{dt}(t)\le0$, so $\lambda_m(\cdot)$ is decreasing on $[0,t_f]$.  Consequently, for every $t\in[0,t_f]$:
	\begin{equation}\label{eq63a}
		\lambda_m(t_f)
		\;\le\;
		\lambda_m(t)
		\;\le\;
		\lambda_m(0),\ \Rightarrow\ -1
		\;\le\;
		\lambda_m(t)
		\;\le\;
		\lambda_m(0). 
	\end{equation}
	and by Eq. (\ref{eq62a}), at $t=0$:
	\begin{equation}\label{eq64}
		-\,\exp\bigl(-a_{\min}t_f\bigr)
		\;\le\;
		\lambda_m(0)
		\;\le\;
		-\,\exp\bigl(-a_{\max}t_f\bigr).
	\end{equation}
Therefore, using Eq. (\ref{eq63a}) and Eq. (\ref{eq64}), we get:
	\begin{equation}\label{eq65}
	-1\le\;
	\lambda_m(t)
	\;\le\;
	-\,\exp\bigl(-a_{\max}t_f\bigr).
\end{equation}
\end{proof}
\subsubsection*{B) Minimum Time Problem:}
Let us consider the scenario where $\mathbf{c}_m = F_z = 0$. In this setting, the objective function is independent of the aircraft’s mass dynamics. Particularly, the objective function is:
\begin{equation}
	\min_{v(t) \in [v_{\min}, v_{\max}],\, \chi(t),\, t_f} \mathcal{J} := t_f.
\end{equation}
Since the mass evolution is decoupled from the objective, it may be excluded from the optimization. Therefore, the minimum-time problem can be solved using only the kinematic equations for $x(t)$ and $y(t)$, while the mass profile can be obtained separately. Hence, we write the Hamiltonian as:
\begin{equation}
	\mathcal{H}(t)=\lambda_xF_x+\lambda_yF_y=\lambda_xF_x+\lambda_x\tan\big(\chi(t)\big)F_y=-\mathbf{c}_t,\quad \forall t\in[0,t_f].
\end{equation} 
\begin{assumption}\label{assumption3}
	$\|\mathbf W\|_2<v(t),\ \forall t\in[0,t_f]$.
\end{assumption}
\begin{theorem}\label{theo3}
	The minimizer $v^*(t)$ of the Hamiltonian for the minimum time problem is $v^*(t)=v_{max},\ \forall t\in[0,t_f]$.
\end{theorem}
\begin{proof}
	The Hamiltonian can be expanded as:
	\begin{equation}\label{eq60}
		\mathcal{H}(t)=\lambda_x\big(v\cos\big(\chi(t)\big)+W_x\big)+\lambda_x\tan\big(\chi(t)\big)+\lambda_x\tan\big(\chi(t)\big)\big(v\sin\big(\chi(t)\big)+W_y\big)=-\mathbf{c}_t.
	\end{equation} 
Since $v^*(t)$ is an affine control in this formulation, we write:
\begin{equation}\label{eq61}
	v^*(t) := \left\{
	\begin{aligned}
		&v_{\max} && S(t) < 0,\\
		&v_{\min} && S(t) > 0,\\
		&\textbf{singular} && S(t) = 0.
	\end{aligned}
	\right.
\end{equation}
 where the switching function $S(t)$ is defined as:
 \begin{equation}
 	S(t):=\frac{\partial\mathcal{H}}{\partial v}=\lambda_x\cos\big(\chi(t)\big)+\lambda_{x}\tan\big(\chi(t)\big)\sin\big(\chi(t)\big)=\frac{\lambda_{x}}{\cos\big(\chi(t)\big)}.
 \end{equation}
On the other hand, we can rearrange the Hamiltonian, Eq. (\ref{eq60}), as:
\begin{equation}\label{eq63}
	\mathcal{H}(t)=\frac{\lambda_x}{\cos\big(\chi(t)\big)}Q(t)=-\mathbf{c}_t<0.
\end{equation}
where $Q(t):=v+W_x\cos\big(\chi(t)\big)+W_y\sin\big(\chi(t)\big)$. 
Now, under Assumption~\ref{assumption3}, we have:
\begin{equation}
	\bigl|W_x\cos\chi(t) + W_y\sin\chi(t)\bigr|
	\;\le\;
	\sqrt{W_x^2 + W_y^2}
	\;=\;\|\mathbf W\|_2
	\;<\;v(t).
\end{equation}
Hence:
\begin{equation}
	Q(t)
	= v(t) + W_x\cos\chi(t) + W_y\sin\chi(t)
	\;\ge\;
	v(t) - \bigl|W_x\cos\chi(t) + W_y\sin\chi(t)\bigr|
	\;>\;0.
\end{equation}
Therefore, from Eq. (\ref{eq63}):
\begin{equation}
	\frac{\lambda_{x}}{\cos\big(\chi(t)\big)}<0 \Rightarrow S(t)<0.
\end{equation}
and from Eq. (\ref{eq61}), $v^*(t)=v_{max},\ \forall t\in[0,t_f]$.
\end{proof}
\begin{theorem}
	Suppose the wind field is constant, $	W_x(x, y) \equiv \bar{W}_x \quad \text{and} \quad W_y(x, y) \equiv \bar{W}_y$
	for some constants \( \bar{W}_x, \bar{W}_y \in \mathbb{R} \). Then the optimal time \( t_f^* \), is:
	\begin{equation}
		t_f^* =\frac{x_f}{v_{max}\cos \chi_0+\bar{W}_x}.
	\end{equation}
where $	\chi_0 = -\arctan\ \frac{x_f}{y_f} + \arccos\ \Bigl(\frac{ x_f\bar W_y - y_f\bar W_x}{v_{\max}\sqrt{x_f^2 + y_f^2}}\Bigr)$.
\end{theorem}

\begin{proof}
	For a constant wind field, from Eq.~(\ref{eq43}) we have \(\chi^*(t)=\chi_0\) constant, and by Theorem~\ref{theo3}, \(v^*(t)=v_{\max}\) for all \(t\in[0,t_f]\).  Hence the state equations reduce to:
	\begin{equation}\label{eq69}
		\begin{split}
		&\frac{dx}{dt} = v_{\max}\cos\chi_0 + \bar W_x,\quad \frac{dy}{dt} = v_{\max}\sin\chi_0 + \bar W_y,\\
		&x(0)=y(0)=0,\quad x(t_f)=x_f,\;y(t_f)=y_f.
		\end{split}
	\end{equation}
	Dividing the two ODEs gives:
	\begin{equation}
		\frac{dy}{dx}
		= \frac{v_{\max}\sin\chi_0 + \bar W_y}{v_{\max}\cos\chi_0 + \bar W_x}
		\;\Longrightarrow\;
		v_{\max}\bigl(y_f\cos\chi_0 - x_f\sin\chi_0\bigr)
		= x_f\bar W_y - y_f\bar W_x.
	\end{equation}
	Let $\mathcal{A} := x_f\bar W_y - y_f\bar W_x$, and $\mathcal{B} := \sqrt{x_f^2 + y_f^2}$. Then:
	\begin{equation}
		y_f\cos\chi_0 - x_f\sin\chi_0
		= \frac{\mathcal{A}}{v_{\max}}
		= \mathcal{B} \,\cos\bigl(\chi_0 + \varphi\bigr),
		\quad
		\varphi := \arctan\!\frac{x_f}{y_f}.
	\end{equation}
	Thus $\cos\bigl(\chi_0 + \varphi\bigr)
	= \frac{\mathcal{A}}{v_{\max}\,\mathcal{B}}$, and since \(\chi_0,\varphi\in[-\tfrac\pi2,\tfrac\pi2]\), we get:
	\begin{equation}
		\chi_0 = -\varphi + \arccos\!\Bigl(\frac{\mathcal{A}}{v_{\max}\mathcal{B}}\Bigr).
	\end{equation}
	which upon substitution yields the stated formula for \(\chi_0\).  Finally, integrating \(\frac{dx}{dt}\) gives $t^*_f$.
\end{proof}

\subsection{The Optimal Boundary Arcs}\label{sec32}
\textbf{a) Control Bounds:} Under regularity assumption for the controls $v(t)$, and $\chi(t)$, any switch to control bounds satisfies $\mathcal{C}_a\big(U(t)\big)=0$, where $\mathcal{C}_a$ is the active entry of $\mathcal{C}(U(t))$ defined as:
\begin{equation}
	\mathcal{C}\bigl(U(t)\bigr):=\bigg(v(t)-v_{\max},\  v_{\min}-v(t),\ \chi(t)-\chi_{\max},\ \chi_{\min}-\chi(t)\bigg)^\top.
\end{equation}
Suppose that $\mathcal{C}_a\big(U(t)\big)=0$ on $t\in [t_1,t_2],\ 0<t_1<t_2<t_f$, then $\frac{\partial \mathcal{H}}{\partial U_a}|_{t\in[t_1,t_2)}\neq 0$ where $U_a\in \{v,\chi\}$. Therefore, under the continuity of the switching function, we have $\frac{\partial \mathcal{H}}{\partial U_a}(t_2)= 0$. 

\textbf{Mixed-Boundary Arc:} The optimal mixed-boundary arc is obtained as the solution to $\mathcal{S}_a\bigl(X,U(t)\bigr)=0$, where $\mathcal{S}_a$ is the active entry of $\mathcal{S}$ defined as:
\begin{equation}
	\mathcal{S}\bigl(X,U(t)\bigr)=\bigg(\Pi(t)-\Pi_{\max},\ \Pi_{\min}-\Pi(t)\bigg)^\top.
\end{equation}
Similarly, suppose that $\mathcal{S}_a\big(X,U(t)\big)=0$ on $t\in [t_1,t_2],\ 0<t_1<t_2<t_f$, then $\frac{\partial \mathcal{H}}{\partial U_a}|_{t\in[t_1,t_2)}\neq 0$ where $U_a\in \{v,\chi\}$. In this case, the optimal mixed-boundary arc $v^*(t)$ is obtained as:
\begin{equation}
	\frac{D(m,v^*)}{T_{\max}(v^*)}-\Pi_b=0,\quad \Pi_b:= \Pi_{\min}\ \lor\ \Pi_{\max}.
\end{equation}
and, under the continuity of the switching function, we have $\frac{\partial \mathcal{H}}{\partial U_a}(t_2)= 0$.  
\subsection{Numerical Solution for the BVP}
The proposed solution methodology is formulated as a boundary value problem, with the shooting parameters defined as $\lambda_{x}(0)$, $\chi(0)$, and the final time $t_f$. The terminal conditions are imposed via the constraints $x(t_f) = x_f$, $y(t_f) = y_f$, and $\lambda_{m}(t_f) = \mathbf{c}_m$. The system dynamics are integrated using a third-order Runge–Kutta method. Initially, the control inputs are assumed to satisfy the first-order optimality condition $\partial \mathcal{H} / \partial U = 0$, and the integration proceeds accordingly until an inequality constraint becomes active. At that point, the control strategy transitions to the appropriate boundary arc, following the analytical framework established in the preceding section. The resulting shooting problem may be addressed using standard gradient-based optimization techniques. In this work, we employed Sequential Quadratic Programming (SQP), implemented via MATLAB's \texttt{fmincon} function, with a trivial objective function, as the primary objective is to enforce satisfaction of the terminal constraints rather than to minimize a cost functional (see Algorithm \ref{alg1}).    
\begin{algorithm}[H]\label{alg1}
	\small
	\caption{Numerical Algorithm of the Surrogate Optimization Framework}\label{alg1}
	\begin{algorithmic}[1]
		\State \textbf{Input:} initial guess $\mathbf{p}^{(0)}=(\lambda_x(0),\chi(0),t_f)$, tolerance $\varepsilon$
		\State $k \gets 0$
		\Repeat
		\State Integrate $\dot{X}=F(X,U)$, where $U$ satisfies $\partial_U \mathcal{H}=0$
		\While{$t<t_f$}
		\If{$\exists\ a$ s.t. $\mathcal{C}_a(U)=0$ or $\mathcal{S}_a(X,U)=0$}
		\State Switch to boundary arc: $ U_a = \begin{cases}
			U_{a,\min} \text{ or } U_{a,\max}, & \text{if } \mathcal{C}_a(U)=0\\
			v^*:\frac{D(m,v^*)}{T_{\max}(v^*)}=\Pi_b, & \text{if } \mathcal{S}_a(X,U)=0
		\end{cases}$
		\While{$|\partial_{U_a} \mathcal{H}|>\delta$}
		\State Continue integration on boundary arc
		\EndWhile
		\State Switch back to $\partial_U \mathcal{H}=0$
		\Else
		\State Continue integration on interior arc
		\EndIf
		\EndWhile
		\State Residual: $\mathbf{r}^{(k)} := \bigl[x(t_f)-x_f,\, y(t_f)-y_f,\, \lambda_m(t_f)-\mathbf{c}_m\bigr]^\top$
		\State Update $\mathbf{p}^{(k+1)}$ via SQP: $\min \frac{1}{2}\|\mathbf{p}-\mathbf{p}^{(k)}\|^2$ s.t. $\mathbf{r}^{(k)}=0$
		\State $k \gets k+1$
		\Until{$\|\mathbf{r}^{(k)}\|<\varepsilon$}
		\State \textbf{Output:} $\mathbf{p}^*, X^*(t), U^*(t)$
	\end{algorithmic}
\end{algorithm}

\section{Simulation Results}\label{sec5}
In this section, we first compare the optimal solutions obtained via the surrogate optimization framework with those generated by a direct transcription method. Subsequently, we present representative optimal trajectories produced by the surrogate framework for several case studies. Notably, across all simulations, boundary arcs, typically characterized by large values of $\mathbf{c}_t$, were observed to be minimal, generally persisting for only a few seconds. Given their negligible duration relative to standard cruise segments, these intervals have been excluded from the comparative illustrations to enhance clarity and focus on the dominant trajectory features.
\subsection{Comparison}
Table \ref{tab:comp} reports a comprehensive set of comparison results between the proposed PMP-based surrogate framework and a high-resolution direct shooting method, in which the system dynamics are discretized using a third-order Runge--Kutta scheme with 300 nodes. The comparisons are carried out across two flight-sensitive regions by varying their associated parameters ($c_{s1}$ and $c_{s2}$), while keeping the remaining parameters fixed. In addition, further experiments were performed by varying the objective coefficients $\mathbf{c}_t$ and $\mathbf{c}_m$. For these latter experiments, the initial mass was fixed at $140$~tons and the cruise altitude at $10$~km. The table also reports cases where the initial mass was varied (150 and 160~tons) and the cruise altitude adjusted (9 and 11~km). 

In all cases, the surrogate framework achieved a substantial computational speedup, averaging approximately $25\times$ faster than the direct shooting method. At the same time, the relative deviation in the optimal objective was found to be negligible, with an average deviation below $0.0004$ (i.e., $0.04\%$). In addition, we compare the constraint violations obtained from the SQP solver in both approaches. Recall that the surrogate framework is formulated as a multi-point boundary value problem; hence, the SQP is applied with only a dummy (trivial) objective function. The comparison is performed for a nominal case with an initial mass of 140~tons, cruise altitude of 10~km, and parameter values $\mathbf{c}_m=-1$, $\mathbf{c}_t=1$, and $c_{s1}=c_{s2}=1$ (see Fig.~\ref{fig:cop1}). As illustrated in the figure, the proposed PMP-based surrogate framework not only converges within a few iterations but also enforces the constraints to several orders of magnitude greater accuracy than the direct shooting method.

\subsection{Case Study: Surrogate Optimization in a Synthetic Wind-Free Scenario}
In this section, we present optimal trajectories and costate evolutions obtained using the surrogate optimization framework for a synthetic scenario with no ambient wind and a circular flight-sensitive region placed along the chord connecting the origin $(0,0)$ to the terminal point $(x_f, y_f = x_f)$. 

To assess the system's response to varying obstacle locations, five test cases are considered, where the center of the circular sensitive region is systematically shifted along the chord from $(0.3\sqrt{2}\,x_f,\,0.3\sqrt{2}\,x_f)$ to $(0.7\sqrt{2}\,x_f,\,0.7\sqrt{2}\,x_f)$ in increments of $0.1\sqrt{2}\,x_f$. Each case is solved as a minimum-fuel optimal control problem, employing terminal cost weights $\mathbf{c}_t = 0$ and $\mathbf{c}_m = -1$, with an initial mass of $m_0 = 150$\,tons, cruising altitude $\bar{h} = 10$\,km, and a terminal target of $(x_f, y_f) = (1000, 1000)$\,km. For each scenario, the obstacle-avoidance weight $\mathbf{c}_s$ is tuned to ensure complete avoidance of the sensitive region.

Figure~\ref{fig:combined} depicts the evolution of state variables, control inputs, and costate dynamics across all five cases. The results reveal a near-symmetric structure in both the state and costate trajectories as the obstacle is displaced along the chord. Of particular interest is the marginal variation observed in the mass-associated costate $\lambda_m(t)$ across all scenarios, suggesting that the mass‐sensitive component of the optimal policy remains largely unaffected by moderate changes in the obstacle's position. In addition, the mass dynamics $m(t)$ remained almost unaffected across all configurations.

{\renewcommand{\arrayrulewidth}{1.3pt} 
	
	\begin{table}[htbp]
		\centering
		\resizebox{\textwidth}{!}{%
			\begin{tabular}{|c|c|p{8cm}|c|c|}
				\hline
				\rowcolor{gray!2} 
				\textbf{Parameter} & \textbf{Value} & \textbf{Fixed Parameters} & \textbf{Relative Error (\%)} & \textbf{CT Ratio} \\
				\hline
				\rowcolor{blue!2} 
				$\mathbf{c}_t$ & $0.1$ & $\mathbf{c}_m=p_m,\; cs_1=p_{s1},\; cs_2=p_{s2},\; m_0=p_0,\; h=p_h$ & $0.052$ & $25.2$ \\
				\hline
				\rowcolor{blue!2} 
				$\mathbf{c}_t$ & $0.2$ & $\mathbf{c}_m=p_m,\; cs_1=p_{s1},\; cs_2=p_{s2},\; m_0=p_0,\; h=p_h$ & $0.041$ & $26.3$ \\
				\hline
				\rowcolor{red!2} 
				$\mathbf{c}_m$ & $-1$ & $\mathbf{c}_t=p_t,\; cs_1=p_{s1},\; cs_2=p_{s2},\; m_0=p_0,\; h=p_h$ & $0.054$ & $24.7$ \\
				\hline
				\rowcolor{red!2} 
				$\mathbf{c}_m$ & $-0.5$ & $\mathbf{c}_t=p_t,\; cs_1=p_{s1},\; cs_2=p_{s2},\; m_0=p_0,\; h=p_h$ & $0.038$ & $25.1$ \\
				\hline
				\rowcolor{green!2} 
				$cs_1$ & $0.5$ & $\mathbf{c}_t=p_t,\; \mathbf{c}_m=p_m,\; cs_2=p_{s2},\; m_0=p_0,\; h=p_h$ & $0.043$ & $27.3$ \\
				\hline
				\rowcolor{green!2} 
				$cs_1$ & $1.5$ & $\mathbf{c}_t=p_t,\; \mathbf{c}_m=p_m,\; cs_2=p_{s2},\; m_0=p_0,\; h=p_h$ & $0.046$ & $27.2$ \\
				\hline
				\rowcolor{orange!2} 
				$cs_2$ & $1$ & $\mathbf{c}_t=p_t,\; \mathbf{c}_m=p_m,\; cs_1=p_{s1},\; m_0=p_0,\; h=p_h$ & $0.039$ & $26.8$ \\
				\hline
				\rowcolor{orange!2} 
				$cs_2$ & $2$ & $\mathbf{c}_t=p_t,\; \mathbf{c}_m=p_m,\; cs_1=p_{s1},\; m_0=p_0,\; h=p_h$ & $0.033$ & $26.1$ \\
				\hline
				\rowcolor{purple!2} 
				$m_0$ & $150 [tons]$ & $\mathbf{c}_t=p_t,\; \mathbf{c}_m=p_m,\; cs_1=p_{s1},\; cs_2=p_{s2},\; h=p_h$ & $0.037$ & $23.4$ \\
				\hline
				\rowcolor{purple!2} 
				$m_0$ & $160 [tons]$ & $\mathbf{c}_t=p_t,\; \mathbf{c}_m=p_m,\; cs_1=p_{s1},\; cs_2=p_{s2},\; h=p_h$ & $0.049$ & $24.5$ \\
				\hline
				\rowcolor{yellow!2} 
				$\bar{h}$ & $9000 [m]$ & $\mathbf{c}_t=p_t,\; \mathbf{c}_m=p_m,\; cs_1=p_{s1},\; cs_2=p_{s2},\; m_0=p_0$ & $0.053$ & $23.9$ \\
				\hline
				\rowcolor{yellow!2} 
				$\bar{h}$ & $11000 [m]$ & $\mathbf{c}_t=p_t,\; \mathbf{c}_m=p_m,\; cs_1=p_{s1},\; cs_2=p_{s2},\; m_0=p_0$ & $0.033$ & $27.1$ \\
				\hline
			\end{tabular}%
		}
		\caption{Comparison of the surrogate model with high-resolution  numerical solution using direct optimization: $p_m=-1, p_{s1}=0.5, p_{s2}=1, p_0=140$ [tons], $p_h=10$ [km], $p_t=0.1$; \textbf{CT Ratio} stands for Computational Time Ratio}
		\label{tab:comp}
	\end{table}
	
}

\begin{figure}[htp]
	\centering
	\subfigure[]{
		\includegraphics[width=0.6\textwidth]{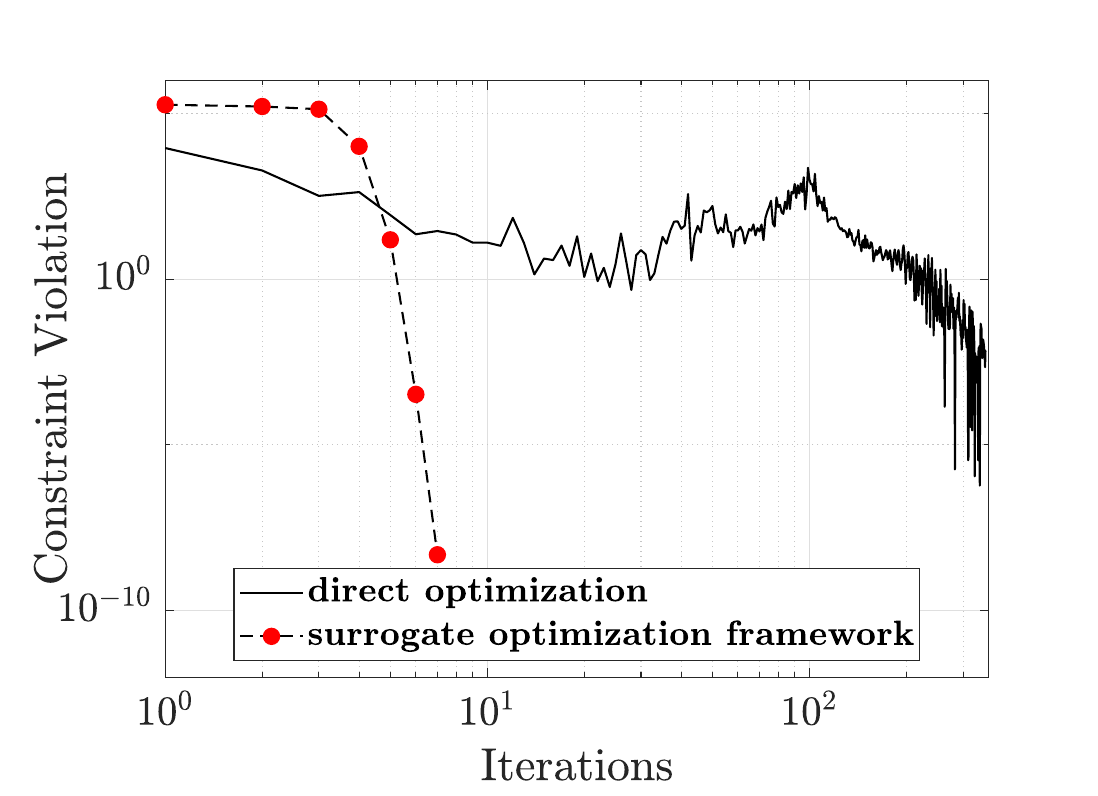}
		\label{fig:dplot}
	}
	\caption{constraint violation versus SQP iteration count: comparison of the \textbf{surrogate} and \textbf{direct} optimization methods}
	\label{fig:cop1}
\end{figure}

\begin{figure}[htp]
	\centering
	\subfigure[]{
		\includegraphics[width=0.48\textwidth]{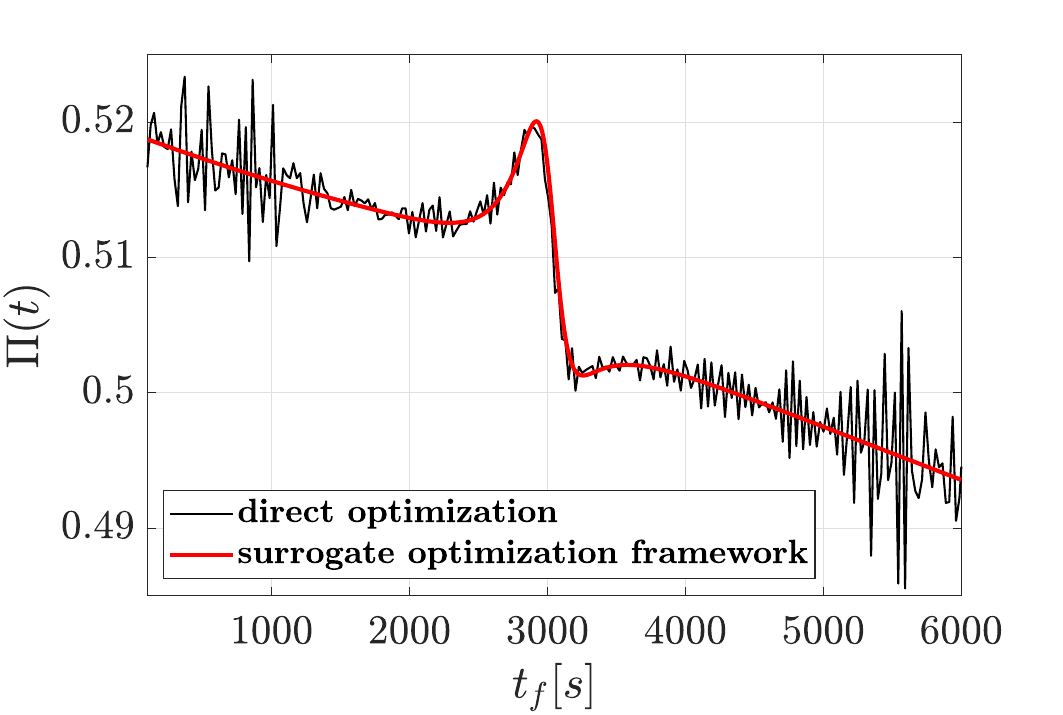}
		\label{fig:schematic}
	}
	\subfigure[]{
		\includegraphics[width=0.48\textwidth]{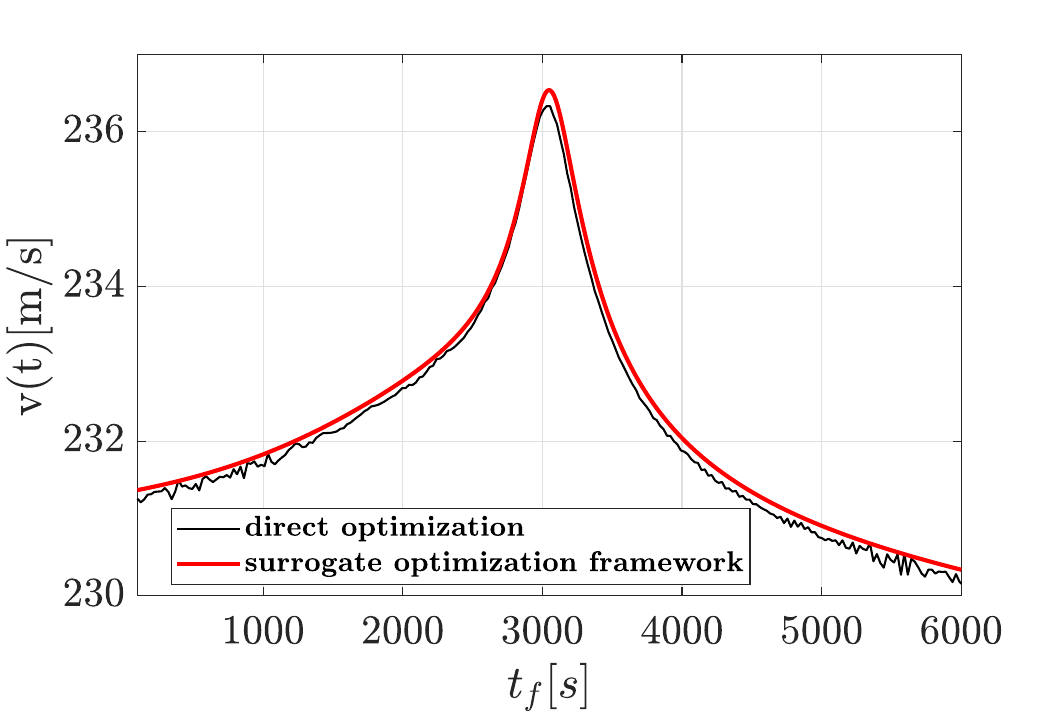}
		\label{fig:dplot}
	}
	\subfigure[]{
	\includegraphics[width=0.48\textwidth]{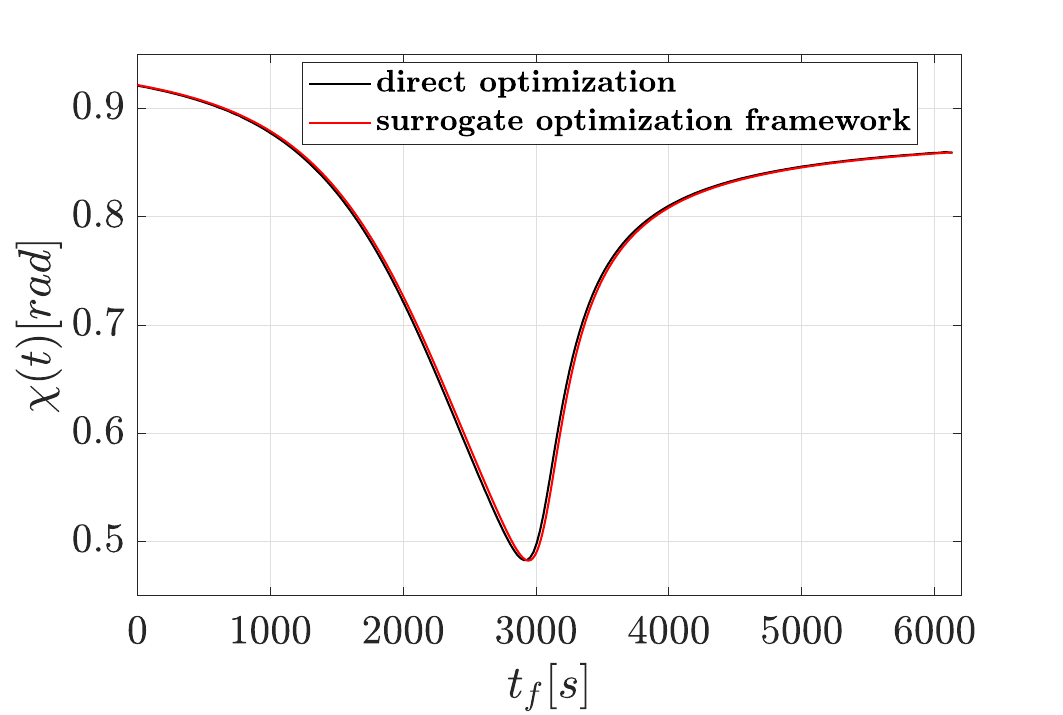}
	\label{fig:dplot}
}
	\subfigure[]{
	\includegraphics[width=0.48\textwidth]{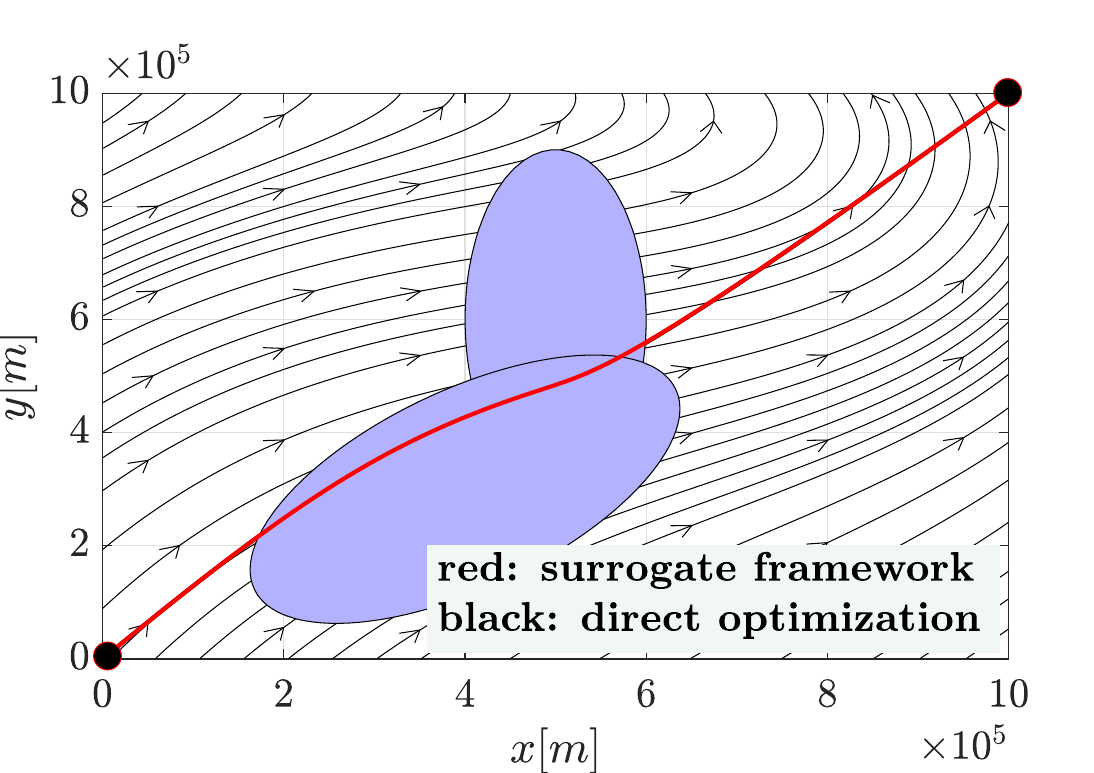}
	\label{fig:dplot}
}

	\caption{comparison between the \textbf{surrogate} optimization and \textbf{direct} method; ellipse$_1$: $a_1=0.1x_f; b_1=0.3x_f; y_{c1}=0.6x_f; x_{c1}=0.5x_f; \alpha_1=0$; ellipse$_2$: $a_2=0.3x_f; b_2=0.15x_f; y_{c2}=0.3x_f; x_{c2}=0.4x_f; \alpha_2=\pi/4$; $\bar{h}=10$ [m]; $m_0=140$ [tons]}
	\label{fig:cop}
\end{figure}
\clearpage
\newgeometry{top=0cm, bottom=0cm, left=3cm, right=3cm}
\begin{figure}[htp]
	\centering
	\subfigure[$\Pi(t)$ as a function of time $t$]{
		\includegraphics[width=0.48\textwidth]{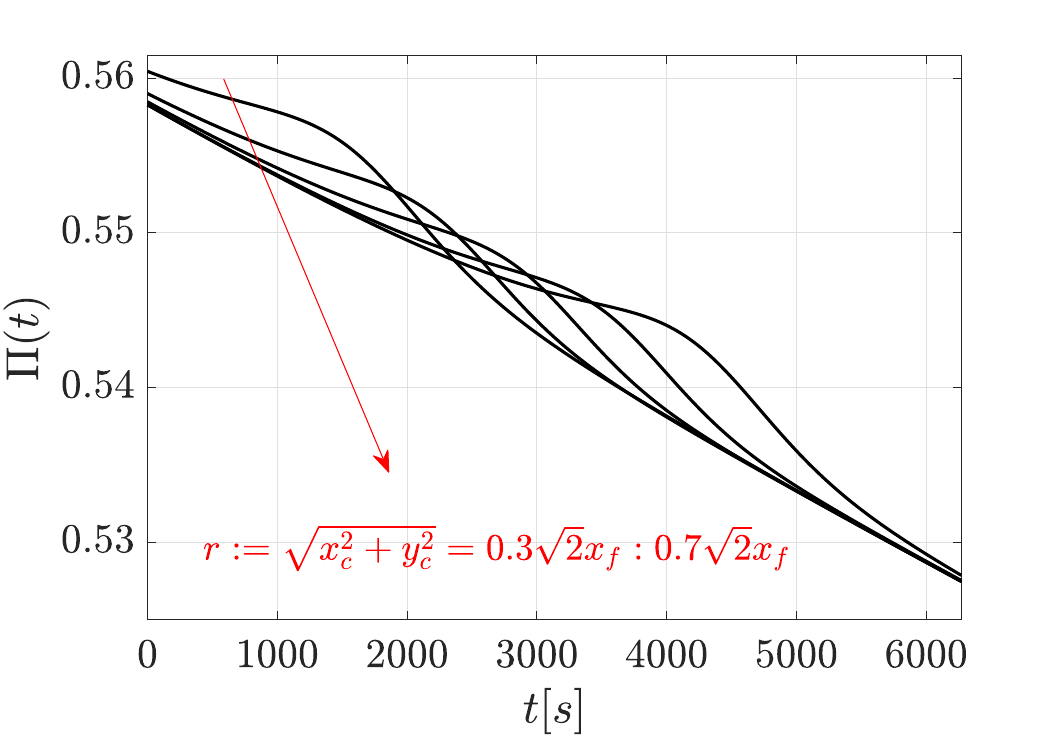}
		\label{fig:schematic}
	} 
	\subfigure[$v(t)$ as a function of time $t$]{
		\includegraphics[width=0.48\textwidth]{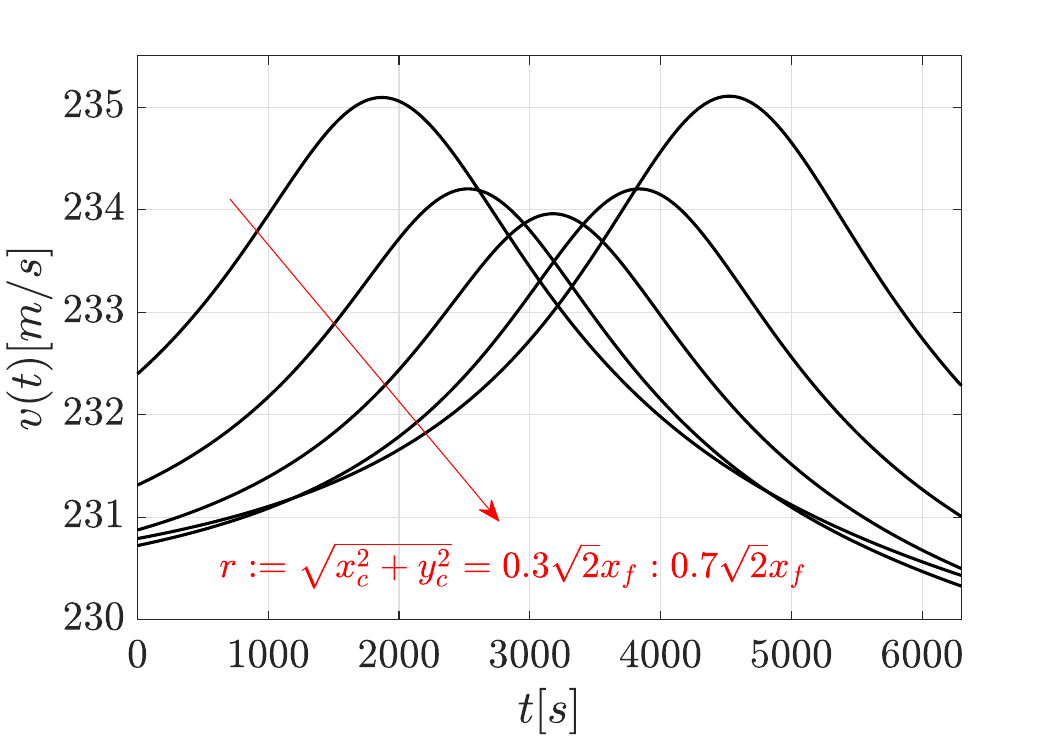}
		\label{fig:dplot}
	}
	\subfigure[$\chi(t)$ as a function of time $t$]{
		\includegraphics[width=0.48\textwidth]{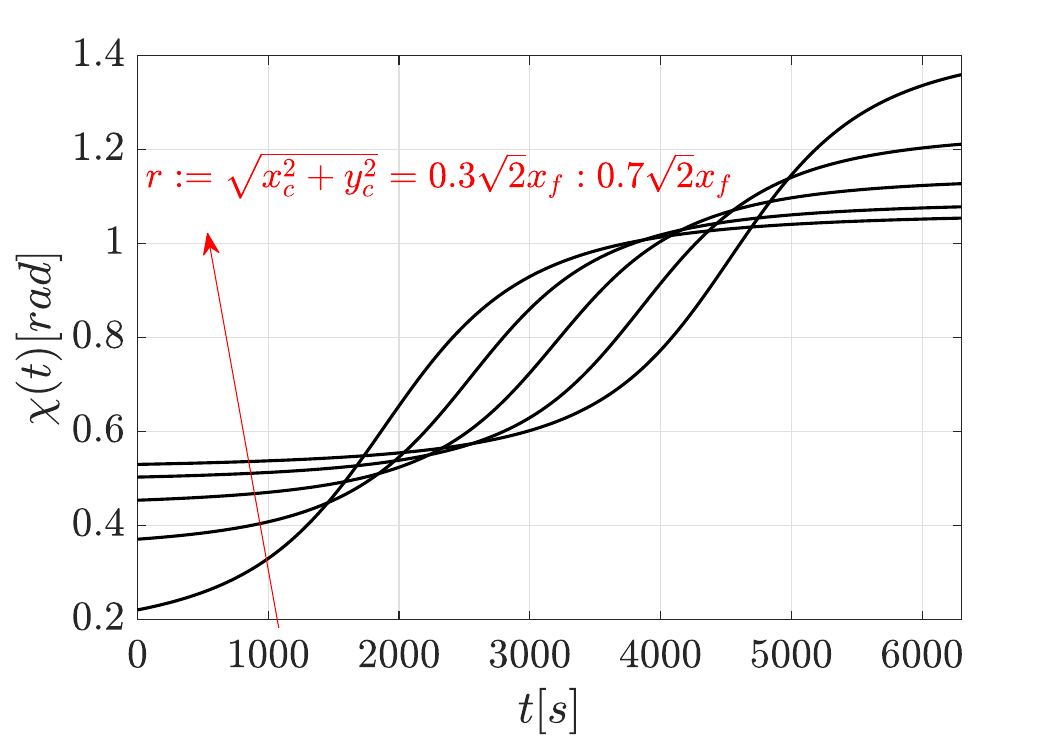}
		\label{fig:dplot}
	}
	\subfigure[lateral trajectories under varying obstacle configurations]{
		\includegraphics[width=0.48\textwidth]{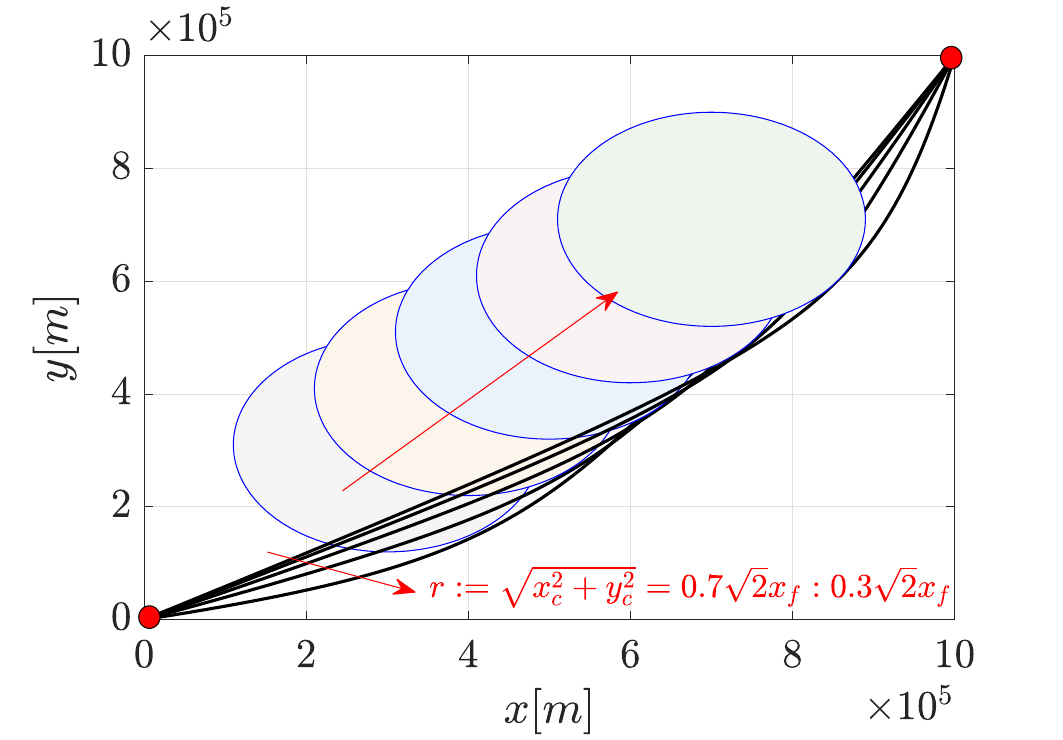}
		\label{fig:dplot}
	}
	\subfigure[costate variable $\lambda_x(t)$ as a function of time $t$]{
		\includegraphics[width=0.48\textwidth]{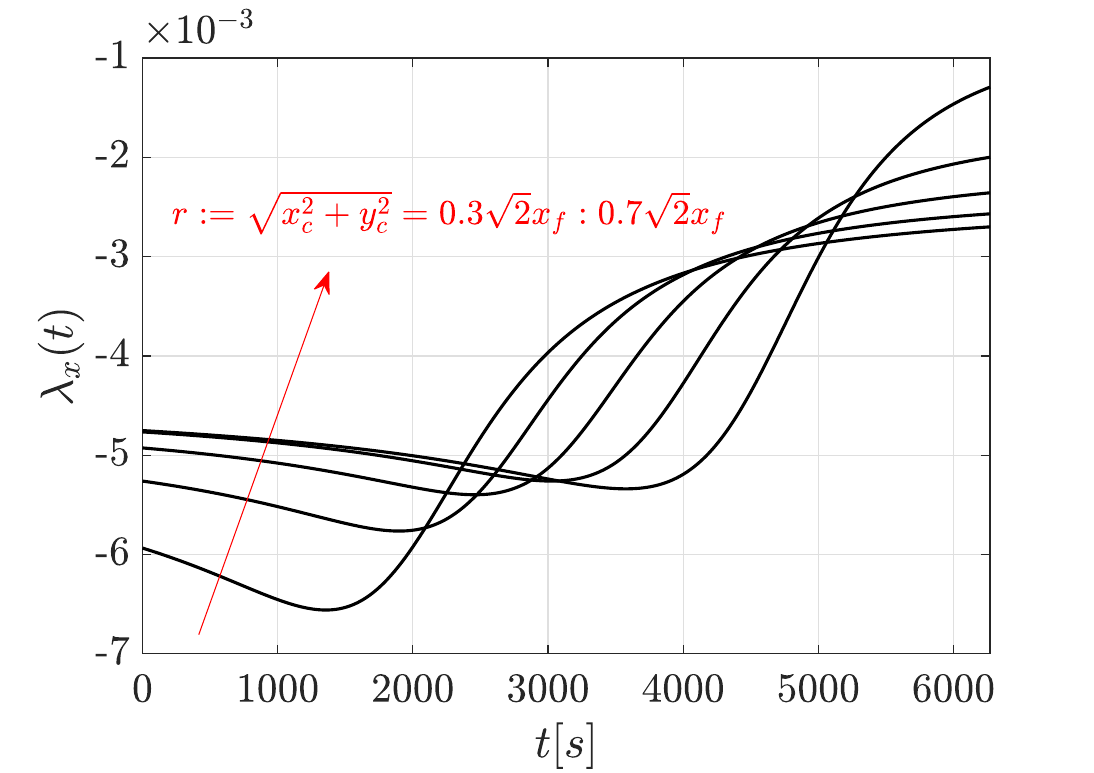}
		\label{fig:dplot}
	}
	\subfigure[costate variable $\lambda_{y}(t)$ as a function of time]{
	\includegraphics[width=0.48\textwidth]{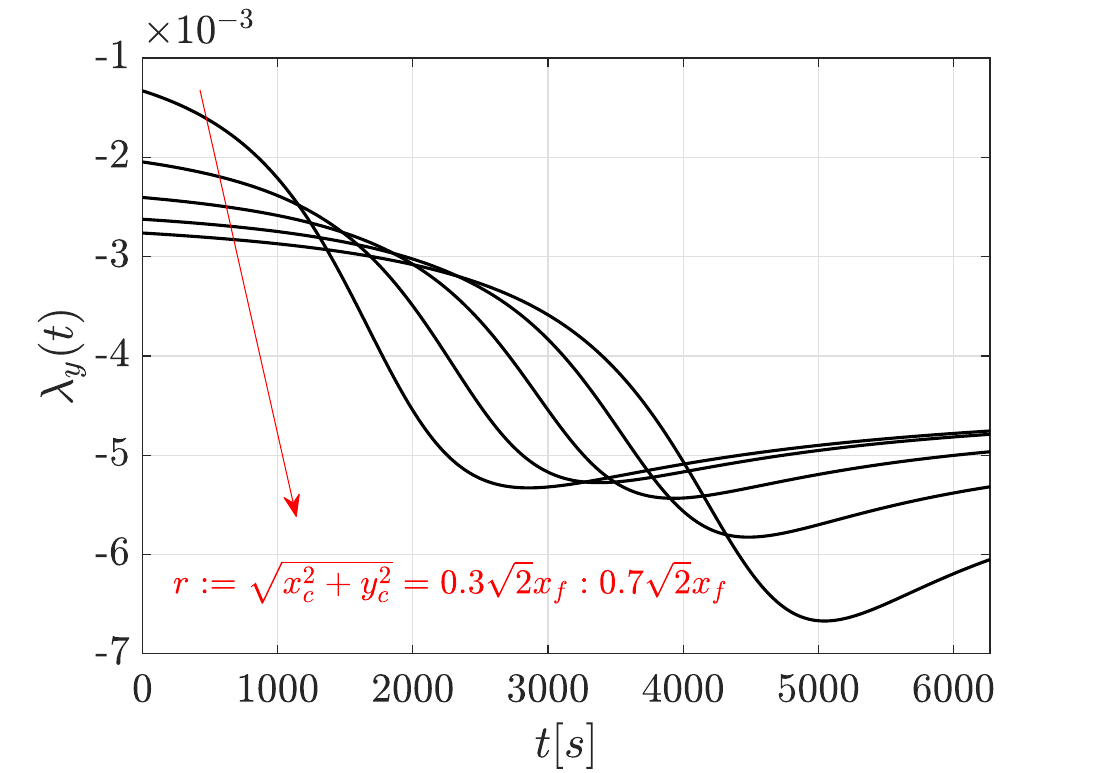}
	\label{fig:dplot}
    }
	\subfigure[costate variable $\lambda_{m}(t)$ as a function of time]{
	\includegraphics[width=0.48\textwidth]{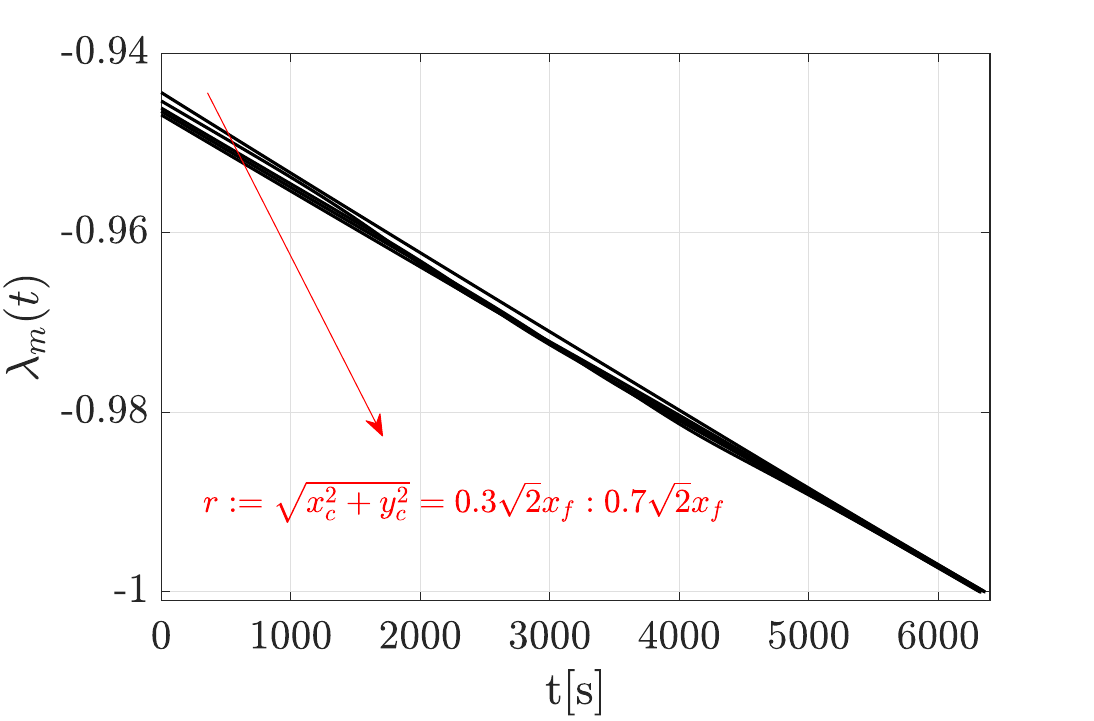}
	\label{fig:dplot}
    }
	\caption{optimal states, costates, and controls for a set of wind-free scenarios by the \textbf{surrogate} optimization framework: $\bar{h}=10 [km], m_0=150 [tons], \mathbf{c}_t=0, \mathbf{c}_m=-1$}
	\label{fig:combined}
\end{figure}
\restoregeometry
\clearpage

\section{Conclusion}\label{sec6}
This work presented a surrogate PMP-based framework for the optimal control of commercial aircraft cruise flights. The framework incorporates spatially variable wind fields and flight-sensitive areas within the optimization process. To model the wind field, we proposed an inviscid composite flow model constructed from reanalysis data. In particular, we argued that the inviscid composite flow formulation naturally satisfies the flow continuity equation while capturing the large-scale coherent structures of atmospheric winds, thereby making it well-suited for wind estimation in trajectory optimization. Flight-sensitive areas were represented by finite ellipses with adjustable parameters, including semi-major and semi-minor axes, orientation, and center position, enabling the flexible modeling of flight-restricted or high-cost regions.  

The proposed surrogate framework demonstrated substantial improvements in computational efficiency compared to high-resolution direct transcription methods. In particular, it achieved significant reductions in solution time while maintaining high accuracy in both states and controls, with virtually no deviation in the total objective value.  

Future work will focus on extending the surrogate formulation to incorporate time-space dependent representations of wind fields and flight-sensitive regions, as well as transitioning to a full three-dimensional aircraft model with all three flight controls. We also note that the costate associated with mass dynamics is not a sensitive variable in the context of commercial cruise optimization. Consequently, this costate variable may be approximated using a predefined linear function (or solved directly), a simplification that could prove valuable in the development of efficient solution methodologies for the full 3D model.


 	
\textbf{Data Availability Statement}\\
Given that the research is theoretical in nature, there is no
relevant data available for sharing.\\

\textbf{Conflict of Interest}\\
The authors declare no conflict of interest of any type within the present submission.

\bibliographystyle{unsrt}
\bibliography{Ref}

\begin{thebibliography}{10}

\bibitem{1}
Daniel González-Arribas, Manuel Soler, Manuel Sanjurjo-Rivo, Maryam
  Kamgarpour, and Juan Simarro.
\newblock Robust aircraft trajectory planning under uncertain convective
  environments with optimal control and rapidly developing thunderstorms.
\newblock {\em Aerospace Science and Technology}, 89:445--459, 2019.

\bibitem{2}
Daniel González-Arribas Abolfazl~Simorgh, Manuel~Soler.
\newblock Robust climate optimal aircraft trajectory planning considering
  uncertainty in weather forecast.
\newblock {\em Proceedings of the 2022 CEAS EuroGNC conference,. Berlin,
  Germany. May 2022. CEAS-GNC-2022-043}.

\bibitem{3}
A.~Simorgh, M.~Soler, D.~Gonz\'alez-Arribas, F.~Linke, B.~L\"uhrs, M.~M.
  Meuser, S.~Dietm\"uller, S.~Matthes, H.~Yamashita, F.~Yin, F.~Castino,
  V.~Grewe, and S.~Baumann.
\newblock Robust 4d climate-optimal flight planning in structured airspace
  using parallelized simulation on gpus: Roost v1.0.
\newblock {\em Geoscientific Model Development}, 16(13):3723--3748, 2023.

\bibitem{4}
Abolfazl Simorgh, Manuel Soler, Daniel González-Arribas, Sigrun Matthes,
  Volker Grewe, Simone Dietmüller, Sabine Baumann, Hiroshi Yamashita, Feijia
  Yin, Federica Castino, Florian Linke, Benjamin Lührs, and
  Maximilian~Mendiguchia Meuser.
\newblock A comprehensive survey on climate optimal aircraft trajectory
  planning.
\newblock {\em Aerospace}, 9(3), 2022.

\bibitem{5}
Amin Jafarimoghaddam and Manuel Soler.
\newblock The indirect bang-singular algorithm (ibsa) for singular control
  problems with state-inequality constraints.
\newblock {\em European Journal of Control}, 2024.

\bibitem{6}
Douglas~M. Pargett and Mark~D. Ardema.
\newblock Flight path optimization at constant altitude.
\newblock {\em Journal of Guidance, Control, and Dynamics}, 30(4):1197--1201,
  2007.

\bibitem{7}
Damián Rivas and Alfonso Valenzuela.
\newblock Compressibility effects on maximum range cruise at constant altitude.
\newblock {\em Journal of Guidance, Control, and Dynamics}, 32(5):1654--1658,
  2009.

\bibitem{9}
Antonio Franco, Damián Rivas, and Alfonso Valenzuela.
\newblock Minimum-fuel cruise at constant altitude with fixed arrival time.
\newblock {\em Journal of Guidance, Control, and Dynamics}, 33(1):280--285,
  2010.

\bibitem{10}
Antonio Franco and Damián Rivas.
\newblock Analysis of optimal aircraft cruise with fixed arrival time including
  wind effects.
\newblock {\em Aerospace Science and Technology}, 32(1):212--222, 2014.

\bibitem{11}
Antonio Franco and Dami\'{a}n Rivas.
\newblock Optimization of multiphase aircraft trajectories using hybrid optimal
  control.
\newblock {\em Journal of Guidance, Control, and Dynamics}, 38(3):452--467,
  2015.

\bibitem{12}
Amin Jafarimoghaddam and Manuel Soler.
\newblock Multi-control commercial aircraft trajectory optimization in a
  vertical plane with state-inequality constraints via singular control theory.
\newblock {\em European Journal of Control}, page 100930, 2023.

\bibitem{8}
Aydan Cavcar and Mustafa Cavcar.
\newblock Approximate solutions of range for constant altitude – constant
  high subsonic speed flight of transport aircraft.
\newblock {\em Aerospace Science and Technology}, 8(6):557--567, 2004.

\bibitem{14}
Hok~K. Ng, Banavar Sridhar, Shon Grabbe, and Neil Chen.
\newblock Cross-polar aircraft trajectory optimization and the potential
  climate impact.
\newblock In {\em 2011 IEEE/AIAA 30th Digital Avionics Systems Conference},
  pages 3D4--1--3D4--15, 2011.

\bibitem{15}
Hok~K. Ng, Banavar Sridhar, and Shon Grabbe.
\newblock Optimizing aircraft trajectories with multiple cruise altitudes in
  the presence of winds.
\newblock {\em Journal of Aerospace Information Systems}, 11(1):35--47, 2014.

\bibitem{13}
Banavar Sridhar, Hok~K. Ng, and Neil~Y. Chen.
\newblock Aircraft trajectory optimization and contrails avoidance in the
  presence of winds.
\newblock {\em Journal of Guidance, Control, and Dynamics}, 34(5):1577--1584,
  2011.

\bibitem{16}
Amin Jafarimoghaddam and Manuel Soler.
\newblock Time-fuel-optimal navigation of a commercial aircraft in cruise with
  heading and throttle controls using pontryagin’s maximum principle.
\newblock {\em IEEE Control Systems Letters}, 7:2970--2975, 2023.

\bibitem{Amin-zermelo}
Amin Jafarimoghaddam and Manuel Soler.
\newblock The applicability of zermelo's equation to indirect and direct
  optimization of commercial aircraft flights.
\newblock {\em Applied Mathematical Modelling}, 135:457--476, 2024.

\bibitem{17}
David~G. Hull.
\newblock Fundamentals of airplane flight mechanics.
\newblock {\em Springer Berlin, Heidelberg, ISBN: 978-3-540-46573-7}, 2007.

\bibitem{18}
Daniel Gonz\'{a}lez-Arribas, Manuel Soler, and Manuel Sanjurjo-Rivo.
\newblock Robust aircraft trajectory planning under wind uncertainty using
  optimal control.
\newblock {\em Journal of Guidance, Control, and Dynamics}, 41(3):673--688,
  2018.

\bibitem{mythesis}
Amin Jafarimoghaddam.
\newblock {\em Perturbed-Analytic Direct Transcription for Optimal Control
  (PADOC) with Application in Commercial Aircraft Trajectory Optimization}.
\newblock PhD thesis, Universidad Carlos III de Madrid, September 2024.

\bibitem{turnpike1}
Timm Faulwasser and Lars Grüne.
\newblock Chapter 11 - turnpike properties in optimal control: An overview of
  discrete-time and continuous-time results.
\newblock In Emmanuel Trélat and Enrique Zuazua, editors, {\em Numerical
  Control: Part A}, volume~23 of {\em Handbook of Numerical Analysis}, pages
  367--400. Elsevier, 2022.

\bibitem{turnpike2}
Emmanuel Trélat and Enrique Zuazua.
\newblock The turnpike property in finite-dimensional nonlinear optimal
  control.
\newblock {\em Journal of Differential Equations}, 258(1):81--114, 2015.

\bibitem{26}
D.H Jacobson, M.M Lele, and J.L Speyer.
\newblock New necessary conditions of optimality for control problems with
  state-variable inequality constraints.
\newblock {\em J. Math. Anal. Appl.}, 35(2):255--284, 1971.

\bibitem{27}
Richard~F. Hartl, Suresh~P. Sethi, and Raymond~G. Vickson.
\newblock A survey of the maximum principles for optimal control problems with
  state constraints.
\newblock {\em SIAM Rev.}, 37(2):181--218, 1995.

\bibitem{19}
Jane's~All the World's Aircraft 2004–2005 edited~by P.~Jackson.
\newblock {\em Jane's Information Group, Coulsdon, England, U.K., 2004, pp.
  608, 610.}

\bibitem{20}
A.~Miele.
\newblock Flight mechanics. theory of flight paths.
\newblock {\em Addison-Wesley,Reading, MA, 1962, pp. 107, 225}.

\bibitem{21}
Heiser W.~H. Mattingly, J.~D. and D.~T. Pratt.
\newblock Aircraft engine design.
\newblock {\em 2nd ed., AIAA Education Series, AIAA, Reston, VA, 2002, pp. 38,
  71}.

\bibitem{22}
Navarro F. A. Nuic~A. Gallo, E. and M.~Iagaru.
\newblock Advanced aircraft performance modeling for atm: Bada 4.0 results.
\newblock {\em Proceedings of the 25th Digital Avionics Systems Conference,
  Inst. of Electrical and Electronics Engineers, Piscataway, NJ, Oct. 2006, pp.
  1–12.}

\end{thebibliography}

\section{Appendix}
\appendix
\section{Appendix}\label{app1}
We give the criteria for nullifying the heading angle dynamic, i.e., $\frac{d\chi}{dt}=\frac{g}{v}\tan\big(\mu(t)\big)$, with nearly zero impact on the optimality. This leads to a simplified, and more consistent cruise-flight dynamics for commercial aircraft, with the control inputs as $\Pi(t)$ and $\chi(t)$.

To start, when analyzing the dynamics of the cruise flight (see \cite{1}), it becomes evident that the heading angle dynamic persists, due to the involvement of the bank angle, $\mu(t)$, in the drag term, $D$, through the lift coefficient, $C_L\propto \frac{1}{\cos(\mu(t))}$.

Putting it into perspective, the lift force is computed as $L=\frac{m(t)g}{\cos(\mu(t))}$. On the other hand, the horizontal component of the lift must adhere to the centrifugal force due to the turning radius, i.e.,: $L\sin(\mu(t))=\frac{m(t)v(t)^2}{r(t)}$ ($r$, being the turning radius). Combining these two, gives $\tan(\mu(t))=\frac{v(t)^2}{r(t)g}$.  

Therefore, if $\mu(t)=O(\epsilon\rightarrow 0)$, $\frac{1}{\cos(\mu(t))}=\sec(\mu(t))=1+\frac{\epsilon^2}{2}+O(\epsilon^4)\approx1, \forall t\in[0,t_f]$; and as a result, $\mu(t)$ becomes isolated in the heading angle dynamic alone, rendering $\chi(t)$ itself the primary active control. 

Notably, the optimal trajectories respecting commercial aircraft in cruise phase are featured by $r^*(t)>>1$, for almost all $t\in[0,t_f]$. More specifically, an optimal curvature in the flight path angle is either due the presence of wind gradients or (momentary) maneuvers over flight-sensitive areas.

Our simulations demonstrate that the fluctuations in $\sec(\mu(t))$ caused by (even) substantial wind gradients are negligible, specifically, $\sec(\mu(t))=\sec\big(\tan^{-1}(\frac{v(t)^2}{r(t)g})\big)\approx1+\frac{\big(\tan^{-1}(\frac{v(t)^2}{r(t)g})\big)^2}{2}\approx1+O(10^{-4}-10^{-3})$. On the other hand, momentary sharp-turn maneuvers are unlikely to significantly affect the overall objective function. Nonetheless, considering bypassing flight-sensitive areas with large radii (e.g., $r(\text{sensitive-area})=O(10^5m)$), even the imposition of hard constraints for flight-sensitive areas cannot give rise to noteworthy errors; because $\sec(\mu(t))=\sec\big(\tan^{-1}(\frac{O(10^4-10^5)}{O(10^6)})\big)\approx1+O(10^{-4}-10^{-3})$ (see e.g., Fig. 12 in \cite{1}, where $\mu(t)$\ in the case of thunderstorm avoidance has been graphed, from which, $\max(\mu(t))\approx 6\ \text{deg.}\rightarrow \sec(\max(\mu(t)))\approx1.005$)

Therefore, the heading dynamics automatically becomes trivial for numerous optimization problems pertaining to commercial aircraft flights.
\section{Appendix}\label{app2}
The compressible drag polar under consideration, $D(m,v,\bar{h})$, adheres to the model presented by Cavcar and Cavcar \cite{8}. Furthermore, we consider Boeing 767-300ER aircraft (a typical twin-engine, wide-body, transport aircraft) with maximum takeoff weight (MTOW) of 186,880 kg  and maximum fuel weight of 73,635 kg \cite{19}. We also consider International Standard Atmospheric (ISA) model, in order to represent $T_{max}(v,\bar{h})$, $C_s(v,\bar{h})$, and $D(m,v,\bar{h})$ as outlined below:
\begin{equation}\label{eq1mm}
	\begin{split}
		&D=\frac{1}{2}\rho(\bar{h}) v^2 s C_D.
	\end{split}
\end{equation}

where,
\begin{equation}\label{eq2mm}
	\begin{split}
		&\rho(\bar{h})=\frac{P(\bar{h})}{\bar{R}\Theta(\bar{h})},\quad P(\bar{h})=P_0\big(\frac{\Theta(\bar{h})}{\Theta_0}\big)^{\frac{g}{\beta \bar{R}}},\quad \Theta(\bar{h})=\Theta_0-\beta \bar{h},\\
		&C_D=\bigg(C_{D_{0,i}}+\sum_{p=1}^{5}k_{0p}\bar{K}^p\bigg)+\bigg(C_{D_{1,i}}+\sum_{p=1}^{5}k_{1p}\bar{K}^p\bigg)C_L+\bigg(C_{D_{2,i}}+\sum_{p=1}^{5}k_{2p}\bar{K}^p\bigg)C_L^2,\\
		&M={v}({c_0 \bar{R} \Theta})^{-\frac{1}{2}},\quad
		\bar{K}:=\bar{K}(M)=
		\begin{cases}
			{(M-0.4)^2}{(1-M^2)^{-\frac{1}{2}}} & M\geq 0.4,\\
			0 & M<0.4.\\
		\end{cases} 
	\end{split}
\end{equation}

%

In the above, $\rho$, $P$, and $\Theta$ represent air density, pressure, and temperature respectively, and $P_0$ is the pressure at sea level. $M$ denotes the Mach number, $s$ delineates the wing surface area, $\beta$ is the temperature lapse factor, $\bar{R}$ stands for the ideal gas constant, $c_0$ is the isentropic expansion factor, $C_{D_{x,i}}$ denotes the incompressible drag coefficients, whereas $k_{xy}$ represents the compressible drag coefficients. Moreover, the lift coefficient is defined as $C_L=\frac{2L}{\rho s v^2}$.   

The maximum thrust ($T_{max}$) and specific fuel consumption are represented by the models presented in \cite{20,21,22} and summarized in \cite{7}:
\begin{equation}\label{eq3m}
	\begin{split}
		&T_{max}=\frac{P(\bar{h})\Theta_0}{P_0\Theta(\bar{h})}T_0\bigg(1+\frac{c_0-1}{2}M^2\bigg)^{\frac{c_0}{c_0-1}}(1-0.49M^{\frac{1}{2}}),\\
		&C_s=C_{s,0}\big(\frac{\Theta}{\Theta_0}\big)^{\frac{1}{2}}(1+1.2M).
	\end{split}
\end{equation}

Table \ref{table1} tabulates the involved coefficients in the above definitions for $D$, $T_{max}$, and $C_s$.

\begin{table}[H]\label{table1}
	\begin{center}
		\rowcolors{2}{lightblue}{white}
		\setlength{\arrayrulewidth}{0.9pt} 
		\fontsize{9}{13}\selectfont
		\caption{Flight Model Constants}
		\label{table1}
		\begin{tabular}{|c||c||c||c||c||c||c||c||c|} 
			\hline
			parameter & value & unit (SI) & parameter & value & unit (SI)&parameter&value&unit(SI)\\
			\hline
			$k_{01}$&0.0067&$-$&$k_{15}$&-2.7672&$-$&$P_0$&101325&$Pa$\\
			\hline
			$k_{02}$&-0.1861&$-$& $k_{21}$&-0.1317&$-$&$c_0$&1.4&$-$\\
			\hline
			$k_{03}$&2.2420&$-$&$k_{22}$&1.3427&$-$&$s$&283.3&$m^2$\\
			\hline
			$k_{04}$&-6.4350&$-$&$k_{23}$&-1.2839&$-$&$T_0$&$5\times10^5$&$N$\\
			\hline
			$k_{05}$&6.3428&$-$&$k_{24}$&5.0164&$-$&$C_{s,0}$&$9\times 10^{-6}$&$kgN^{-1}s^{-1}$\\
			\hline
			$k_{11}$&0.0962&$-$&$k_{25}$&0.0000&$-$&$\bar{R}$&287.04&$J.kg^{-1}K^{-1}$\\
			\hline
			$k_{12}$&-0.7602&$-$&$C_{D_{0,i}}$&0.01322&$-$&$g$&9.81&$m/s^2$\\
			\hline
			$k_{13}$&-1.2870&$-$&$C_{D_{1,i}}$&-0.00610&$-$&$\Theta_0$&288.15&$K$\\
			\hline
			$k_{14}$&3.7925&$-$&$C_{D_{2,i}}$&0.06000&$-$&$\beta$&0.0065&$K.m^{-1}$\\
			\hline
		\end{tabular}
	\end{center}
\end{table}

\newpage
	\begin{center}
	{\fontsize{15}{24}\selectfont \textbf{A Surrogate Framework for Cruise-Phase Optimization of Commercial Aircraft: Supplementary Material}}
	
	\vspace{0.5cm}
	
	\textbf{Amin Jafarimoghaddam}$^\star$\footnote{$^\star$Corresponding Author.}, \textbf{Manuel Soler}$^{\dagger}$, \textbf{María Cerezo-Magaña} $^{\dagger}$\\	
	$^{\star,\dagger}${\em{Department of Aerospace Engineering, Universidad Carlos III de Madrid. Avenida de la Universidad, 30, Leganes, 28911 Madrid, Spain
	}}\footnote{Email addresses: ajafarim@pa.uc3m.es (A. Jafarimoghaddam), masolera@ing.uc3m.es (M. Soler), mcerezo@ing.uc3m.es (M. Cerezo)}	\\	
\end{center}

\vspace{0.5cm}

\section{Introduction}
This supplementary document provides a detailed account of the mathematical formulations underlying the MATLAB implementation of the \emph{Surrogate Framework} for cruise-phase optimization of commercial aircraft, together with an example, targeting wind-induced variability in optimal solutions.

In \textbf{Section~D}, we describe the aerodynamic models, environmental inputs, and optimal control problem (OCP) embedded in the code. The section further outlines the modeling of flight-sensitive areas, wind fields, drag and thrust dynamics, fuel consumption, and the atmospheric model. We then present the surrogate OCP formulation, its associated boundary-value problem, and the derivation of the optimal aircraft speed (\href{https://github.com/Amin-M-Jafari/A-Surrogate-Framework-for-General-Cruise-Phase-Optimization-of-Commercial-Aircraft.git}{\textcolor{blue}{\textbf{Link to the General Code}}}).  

In \textbf{Section~E}, we analyze wind-induced variability in optimal solutions via a Monte Carlo study based on the surrogate framework and the composite inviscid wind flow model. The latter is argued to capture large-scale atmospheric flow structures while satisfying the continuity equation, thereby producing realistic wind ensembles.

\section{MATLAB Code for Cruise-Phase Optimization}

The MATLAB code facilitates cruise-phase optimization by integrating aerodynamic and environmental parameters. Users may either supply essential inputs, such as flight-sensitive areas, wind functions, and models for drag, thrust, and fuel consumption, or alternatively use the interactive built-in environment to define these components. The following subsections provide detailed descriptions of these elements.

\subsection{Flight-Sensitive Areas and Clustering}
The code accepts scattered points representing flight-sensitive areas and clusters them into oblique ellipses according to a user-defined number of clusters. Each area is defined by:
\begin{equation}
	\begin{split}
		&		G_{s,i}(x,y) = \left\| X_s - X_{sc,i} \right\|_{\mathbf{A}},\quad X_s =
		\begin{pmatrix}
			x\\[0.2cm]
			y
		\end{pmatrix}, \quad 
		\left\| \bullet \right\|_{\mathbf{A}} = \sqrt{(\bullet)^{\top} \mathbf{A} (\bullet)},
	\end{split}
\end{equation}
where the metric matrix $\mathbf{A}$ is defined as:
\begin{equation}
	\mathbf{A} = \mathbf{R} \mathbf{D} \mathbf{R}^{\top}, \quad 
	\mathbf{D} =
	\begin{pmatrix}
		\frac{1}{a_i^2} & 0\\[0.2cm]
		0 & \frac{1}{b_i^2}
	\end{pmatrix}, \quad 
	\mathbf{R} =
	\begin{pmatrix}
		\cos(\alpha_i) & -\sin(\alpha_i)\\[0.2cm]
		\sin(\alpha_i) & \cos(\alpha_i)
	\end{pmatrix}.
\end{equation}

\subsubsection{Clustering Method}
Given a set of points \( \{X_{s,j}\}_{j=1}^{N} \) and a prescribed number of clusters \( K \), the clustering method proceeds as follows:
\begin{enumerate}
	\item \textbf{Clustering:} Using a clustering algorithm (i.e., \( k \)-means) to partition the points into \( K \) clusters.
	\item \textbf{Centroid Computation:} For each cluster \( i \), the code computes the centroid \( X_{sc,i} \).
	\item \textbf{Covariance and Eigenvalue Decomposition:} Determining the covariance matrix \( \Sigma_i \) and performing eigenvalue decomposition: $			\Sigma_i = \mathbf{R}_i 
	\begin{pmatrix}
		\lambda_{i1} & 0 \\[0.2cm]
		0 & \lambda_{i2}
	\end{pmatrix}
	\mathbf{R}_i^{\top},$
	where \( \lambda_{i1}, \lambda_{i2} \) are the eigenvalues and \( \mathbf{R}_i \) is the rotation matrix.
	\item \textbf{Ellipse Construction:} We set the ellipse semi-axes as $ a_i = k\sqrt{\lambda_{i1}},$ and $b_i = k\sqrt{\lambda_{i2}}$, where \( k \) is a scaling factor ensuring coverage of the points.
	\item \textbf{Orientation:} The ellipse orientation is given by $			\alpha_i = \tan^{-1}\left(\frac{\mathbf{R}_{i,21}}{\mathbf{R}_{i,11}}\right)$.
\end{enumerate}

\subsection{Wind Functions}
The user may either provide the wind velocity components in the \(x\)- and \(y\)-directions, denoted by \(W_x(x,y)\) and \(W_y(x,y)\), or employ the interactive built-in environment for this purpose. In the latter case, the user specifies only the absolute maximum wind magnitude, and the code automatically generates and visualizes random wind fields. These fields are constructed using the composite inviscid wind flow model, as described in the paper.
\subsection{Aircraft Drag Model}
The drag coefficient is modeled as:
\begin{equation}
	C_D = C_D(C_L, M, v) = a(M,v)C_L^2 + b(M,v)C_L + c(M,v),
\end{equation}
where \( C_L \) is the lift coefficient, defined as $C_L=\frac{2mg}{\rho s v^2}$, and \( M \) is the Mach number, defined as $M=\frac{v}{\sqrt{c_0\,\bar{R}\,\Theta}},$ and \( v \) is the aircraft speed.

The functions \( a(M,v) \), \( b(M,v) \), and \( c(M,v) \) can be provided by the user. However, if the user switches to the built-in environment, the default setting is the compressible polar drag model by Cavcar and Cavcar as detailed in the Appendix-B of the paper.

\subsection{Thrust and Fuel-Rate Models}
The user may either input the thrust and fuel-rate models as functions of Mach number and speed, given by \(T_{\max} = T_{\max}(M,v)\) and \(C_s = C_s(M,v)\), or switch to the default mode, which employs the thrust and fuel-rate models as described in the Appendix-B of the paper.

\subsection{Cruise Altitude and Atmospheric Model}
The user must specify the scalar cruise altitude \( \bar{h} \), which is then automatically used within the International Standard Atmosphere (ISA) model to determine the corresponding temperature, pressure, and air density.

\subsection{The Surrogate Optimization Framework}
As detailed in the paper, the surrogate optimization framework for the objective function $		\min_{v(t),\chi(t),t_f} \; \mathcal{J} := \mathbf{c}_t\,t_f + \mathbf{c}_m\,m_f + z_f$ (omitting mixed-inequality constraints) reduces to the following BVP:
%
%
%
%


\begin{equation}\label{eq:ode_system}
	\begin{split}
		&\frac{dx}{dt} = v(t)\frac{1}{\sqrt{1+q^2}} + W_x(x,y) =: F_x,\qquad \frac{dy}{dt} = v(t)\frac{q}{\sqrt{1+q^2}} + W_y(x,y) =: F_y,\\[0.2cm]
		&\frac{dm}{dt} = -\frac{1}{2}\rho s\, v(t)^2 C_D C_s =: F_m,\qquad \frac{dz}{dt} = g(x,y) =: F_z,\\[0.2cm]
		&\frac{d\lambda_{x}}{dt} = -\frac{\partial g}{\partial x} - \lambda_{x}\Bigl(\frac{\partial W_x}{\partial x} + q\frac{\partial W_y}{\partial x}\Bigr),\\[0.2cm]
		&\frac{dq}{dt} = -\frac{\partial W_x}{\partial y} + \Bigl(\frac{\partial W_x}{\partial x} - \frac{\partial W_y}{\partial y}\Bigr)q + \frac{\partial W_y}{\partial x}\,q^2 + \frac{1}{\lambda_{x}}\Bigl(q\frac{\partial g}{\partial x} - \frac{\partial g}{\partial y}\Bigr).
	\end{split}
\end{equation}
where $g(x,y) := \sum_{i=1}^{n} \frac{\mathbf{c}_{s,i}}{G_{s,i}}$, and $q=\tan\big(\chi(t)\big)$. The initial conditions are:
\begin{equation}\label{eq:init_cond}
	\begin{aligned}
		x(0) &= 0, \quad y(0)=0, \quad m(0)=m_0, \quad z(0)=0.
	\end{aligned}
\end{equation}
BVP is formed with three shooting parameters $\lambda_x(0)$, $q(0)$, and $t_f$ and three closing equations at \( t_f \):
\begin{equation}\label{eq:closing}
	\begin{split}
		&x(t_f)= x_f,\qquad y(t_f)= y_f,\\[0.2cm]
		\lambda_m(t_f) &= \frac{-\mathbf{c}_t - F_z(t_f) - \lambda_x(t_f)\bigl(F_x(t_f) + q(t_f)F_y(t_f)\bigr)}{F_m(t_f)} = \mathbf{c}_m.
	\end{split}
\end{equation}

The optimal aircraft speed \( v(t) \) is determined implicitly by solving:
\begin{equation}\label{eq:opt_v}
	\lambda_x\sqrt{1+q^2} - \Biggl(\mathbf{c}_t + F_z + \lambda_x F_x + \lambda_xq\,F_y\Biggr)\left(\frac{1}{C_s}\frac{\partial C_s}{\partial{v}} + \frac{2}{v} + \frac{1}{C_D}\frac{\partial C_D}{\partial{v}}\right)=0.
\end{equation}

\textbf{Note:} After the user provides the model inputs, the implementation automatically promotes the necessary quantities to symbolic variables and uses symbolic differentiation to compute the derivatives appearing in the state dynamics and in the closed-form optimal-speed formula.

\subsection{Post-Processing and Visualization}
After obtaining the optimal trajectories, the MATLAB code plots the evolution of the state variables \( x(t) \), \( y(t) \), \( m(t) \), \( q(t) \), \( z(t) \), and the costate \( \lambda_x(t) \) in separate figures. Moreover, the throttle setting \( \Pi(t) \) is reconstructed at each time step using:
\begin{equation}
	\Pi(t)=\frac{m\frac{dv}{dt}+D}{T_\text{max}}.
\end{equation}


\section{Stochastic Analysis of the Wind-Induced Variability in Optimal Solution}
Without loss of generality, in the presence of wind, the aircraft's overall strategy is maximizing access to the strongest favorable wind while minimizing deviation from the straight line. Therefore, analyzing the influence of wind, referred to as \textit{wind gain} within the minimum-time framework reveals the fundamental contribution of wind to minimizing both time and fuel objectives. 

The purpose of the following analysis is to assess the extent to which the optimal solution obtained with the average wind field over the entire domain deviates from the optimal solution when the full space-dependent wind field is considered. To this end, we derive PDFs for both cases and demonstrate that the probability of $\tfrac{t_{f,\text{rand}}}{t_{f,\text{average}}}$ exceeding a deviation of about 4\% is negligible (even for extreme cases), where $t_{f,\text{rand}}$ corresponds to the optimal final time obtained with random wind fields, and $t_{f,\text{average}}$ denotes the optimal final time with the average of the same wind fields.  

Furthermore, we introduce a bandwidth averaging strategy that concentrates the probability distribution around its center. Consequently, we argue that replacing the space-dependent wind field with the bandwidth-averaged wind of the same field leaves the optimal objective function nearly unaffected. This finding is particularly useful in scenarios where incorporating the full space-dependent wind field is not feasible or computationally efficient.

\subsection{Bandwidth Averaging Strategy}
Let the straight‐line chord from \((0,0)\) to \((L,L)\) have length $L_0 =\sqrt{2}\,L$; and let \(\gamma_s\colon[0,1]\to\Omega\) be a time‐optimal curve of length: 
\begin{equation}
	S \;=\;\int_{0}^{1}\bigl\|\gamma_s'(\tau)\bigr\|\,d\tau.
\end{equation}
If the aircraft’s still‐air speed is \(v_0>0\), and we assume (in the best‐case scenario) a constant tailwind \(+W_{s,\max}\) along the curved path \(\gamma_s\) and a constant headwind \(-\,W_{L_0,\min}\) along the straight‐line chord, then the corresponding flight times satisfy:
\begin{equation}
	T_s \;=\;\frac{S}{\,v_0 + W_{s,\max}\,}, 
	\qquad 
	T_0 \;=\;\frac{L_0}{\,v_0 - W_{L_0,\min}\,}.
\end{equation}
Requiring the wind‐optimal path to be faster than the straight‐line path yields:
\begin{equation}
	\frac{S}{\,v_0 + W_{s,\max}\,}
	\;<\;
	\frac{L_0}{\,v_0 - W_{L_0,\min}\,}
	\;\;\Longrightarrow\;\;
	r \;:=\;\frac{S}{\,L_0\,}
	\;<\;
	\frac{\,v_0 + W_{s,\max}\,}{\,v_0 - W_{L_0,\min}\,}.
\end{equation}
In a symmetric geometry one may set \(W_{s,\max} = W_{L_0,\min} = \max_{(x,y)\in\Omega}\|\mathbf{W}(x,y)\|\), but since the probability of this event is nearly zero, a more realistic assumption is to replace these extremal values by average values over $\Omega$ by defining:
\begin{equation}
	\overline{W}_{\Omega} 
	\;=\; 
	\frac{1}{|\Omega|}\iint_{\Omega}\|\mathbf{W}(x,y)\|\,dx\,dy.
\end{equation}

therefore, we can write the upper bound for the effective ratio $r:=\frac{S}{L_0}$ as: $	r \;=\;\frac{\,1 + \frac{\overline{W}_{\Omega}}{\,v_{0}\,}\,}{\,1 - \frac{\overline{W}_{\Omega}}{\,v_{0}\,}\,}.$

Classical extremal theory shows that among all planar curves of fixed length joining two points, the circular arc maximizes the maximum distance from the chord. Specifically, \(r\) (which is obtained using the wind data) determines a unique central angle \(\theta\in(0,\pi)\) satisfying: $\frac{\theta}{2\sin(\theta/2)} \;=\; r.$

The maximum half‐width \(h\) of the corresponding circular arc above the chord is:
\begin{equation}
	h \;=\;\frac{d}{2\sin(\theta/2)}\Bigl(1 - \cos(\tfrac{\theta}{2})\Bigr).
\end{equation}
and hence the total bandwidth is $D =2\,h$ (see Fig. \ref{fig:band}\textbf{-a}, and Fig. \ref{fig:band}\textbf{-b}). 

\begin{figure}[H]
	\centering
	\subfigure[]{
		\includegraphics[width=0.42\textwidth]{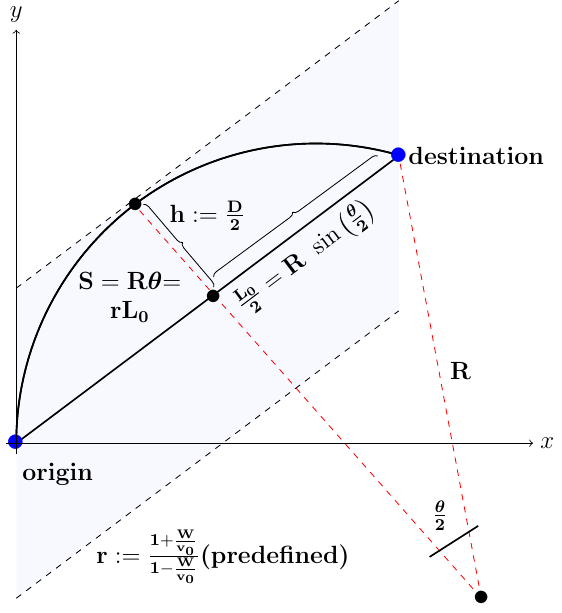}
		\label{fig:schematic}
	}
	\subfigure[]{
		\includegraphics[width=0.5\textwidth]{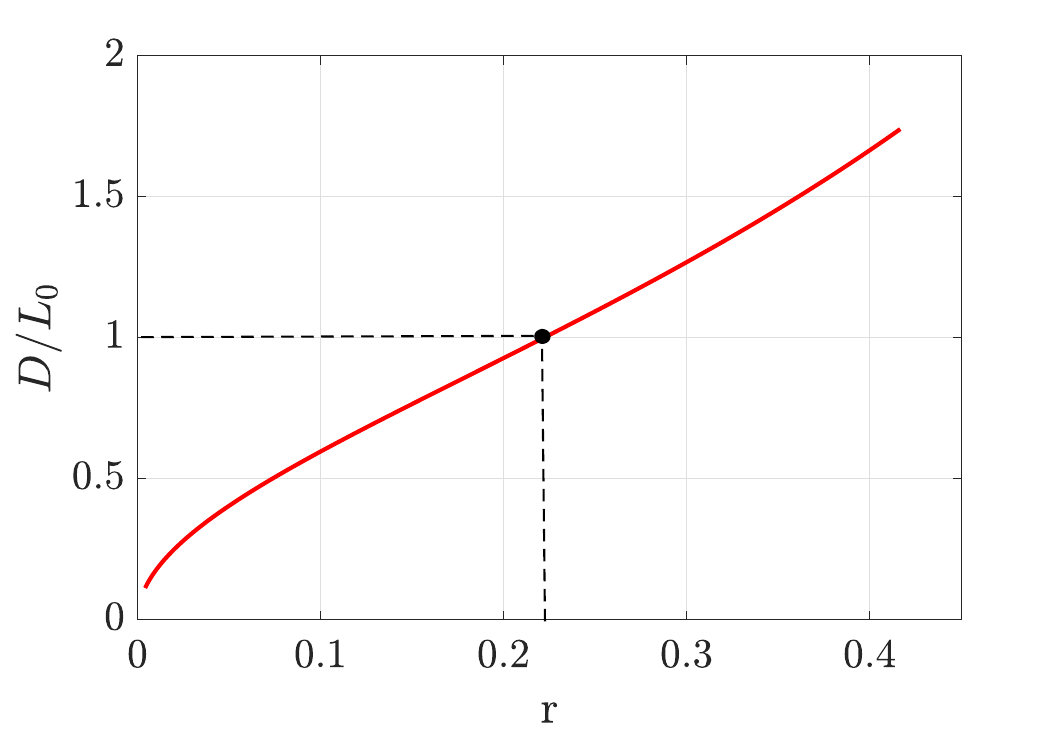}
		\label{fig:dplot}
	}
	\caption{a) Diagram of bandwidth-based average wind area, b) $\frac{D}{L_0}$ as a function of $r$}
	\label{fig:band}
\end{figure}

\subsection{Monte Carlo Simulation and Results}    
\noindent Let \((\Gamma,\mathcal{F},\mathbb{P})\) be the underlying probability space.  We assume given a measurable mapping: 
\begin{equation}
	\omega \;\mapsto\; \mathcal{W}(\,\cdot\,;\omega)\;=\;\bigl(W_{x}(x,y;\omega),\,W_{y}(x,y;\omega)\bigr)\in L^\infty(\Omega;\mathbb{R}^2).
\end{equation}
where \(\Omega=[0,L]\times[0,L]\).  We fix a constant \(W_{\max}>0\) and define the scaled field:
\begin{equation}
	\widetilde{\mathcal{W}}(x,y;\omega)
	=\frac{W_{\max}}{\lVert \mathcal{W}(\,\cdot\,;\omega)\rVert_{L^\infty}}\,\mathcal{W}(x,y;\omega).
\end{equation}
The average wind vector over the entire domain is written as:
\begin{equation}
	\overline{W}_{\Omega}(\omega)
	=\bigl\langle \sqrt{\widetilde{\mathcal{W}}_{x}(x,y;\omega)^2 + \widetilde{\mathcal{W}}_{y}(x,y;\omega)^2}\bigr\rangle_{\Omega}.
\end{equation}
where $\langle f\rangle_{(.)}
:=\frac{1}{|(.)|}\int\!\!\int_{(.)} f(x,y)\,dx\,dy$, for any scalar $f$, and $|(.)|$ is the area (Lebesgue measure) of $(.)$.

%
The effective ratio is then set to:
\begin{equation}
	r \;=\;\frac{\,1 + \frac{\overline{W}_{\Omega}(\omega)}{\,v_{0}\,}\,}{\,1 - \frac{\overline{W}_{\Omega}(\omega)}{\,v_{0}\,}\,}.
\end{equation}
Band‐width radius is obtained by solving the implicit equation:
\begin{equation}
	G(\Theta;\omega)
	=\frac{\Theta}{2\,\sin(\Theta/2)}
	-r,
	\qquad \Theta\in(0,\pi).
\end{equation}
Under mild hypotheses, there is a unique root \(\Theta^{*}(\omega)\in(0,\pi)\) with $G\bigl(\Theta^{*}(\omega);\omega\bigr)=0$. Therefore, the band‐width distance is obtained as: 
\begin{equation}
	D(\omega)
	=\frac{2\,L}{2\,\sin\!\bigl(\Theta^{*}(\omega)/2\bigr)}\Bigl(1 - \cos\!\bigl(\tfrac{\Theta^{*}(\omega)}{2}\bigr)\Bigr).
\end{equation}
giving the masked area as:
\begin{equation}
	\mathcal{B}(\omega)
	=\{(x,y)\in\Omega :|\,y - x\,|\le b(\omega)/2\},\quad b(\omega)=\frac{D(\omega)}{\cos(\pi/4)}.
\end{equation}
Therefore, the average wind within the band-width is written as:
\begin{equation}
	\overline{W}_b(\omega)
	=\bigl\langle \sqrt{\widetilde{\mathcal{W}}_{x}(x,y;\omega)^2 + \widetilde{\mathcal{W}}_{y}(x,y;\omega)^2}\bigr\rangle_{\mathcal{B}(\omega)}.
\end{equation}
The representative boundary value problem is considered as:
\begin{equation}\label{eq1}
	\begin{split}
		&\frac{dx}{dt}=v_{0}\cos(\chi(t))+\widetilde{\mathcal{W}}_x(x,y,\omega),\\
		&\frac{dy}{dt}=v_{0}\sin(\chi(t))+\widetilde{\mathcal{W}}_y(x,y,\omega),\\
		&\frac{dq}{dt}=-\frac{\partial \widetilde{\mathcal{W}}_x}{\partial y}+\Bigl(\frac{\partial \widetilde{\mathcal{W}}_x}{\partial x}-\frac{\partial \widetilde{\mathcal{W}}_y}{\partial y}\Bigr)q+\Bigl(\frac{\partial \widetilde{\mathcal{W}}_y}{\partial x}\Bigr)q^2,\quad q:=\tan\bigl(\chi(t)\bigr),\\
		& x(0)=y(0)=0,\quad x(t_f)=y(t_f)=L,\quad q(0)=unknown,\quad t_f=unknown
	\end{split}
\end{equation}
The solution is obtained by solving:
\begin{equation}
	\mathcal{F}(t_f,q(0);\omega)
	=\bigl(x(t_f;q(0),\omega)-L,\;y(t_f;q(0),\omega)-L\bigr)=(0,0).
\end{equation}
We sample \(\{\omega_{k}\}_{k=1}^{N}\) i.i.d.\ from \(\mathbb{P}\) and  for each \(k\), compute:
\begin{equation}
	t_{k}
	:= t_{f}(\omega_{k}),
	\qquad
	t_{k}^{\mathrm{avg}}
	:= t_{f}^{\mathrm{avg}}(\omega_{k}).
\end{equation}
%

Fig.~\ref{fig:monte} shows the PDF as a function of the relative deviation for different values of the index $p := \frac{\max_{\Omega}(||\mathbf{W}||)}{v_0}$, which represents the relative magnitude of the wind field. Clearly, as $p$ increases, the PDFs become wider, allowing for more probable deviations from the average wind field. Nonetheless, we observed that if the wind is averaged over the proposed bandwidth, the PDF remains narrow even for larger $p$ values. In other words, these figures suggest that the space-dependent wind can be safely replaced with an average wind (constant wind field) if the averaging is performed within the bandwidth domain. This finding is particularly useful for optimal control applications where analyzing the full space-dependent wind field is not computationally efficient.

\begin{figure}[htp]
	\centering
	\subfigure[]{
		\includegraphics[width=0.485\textwidth]{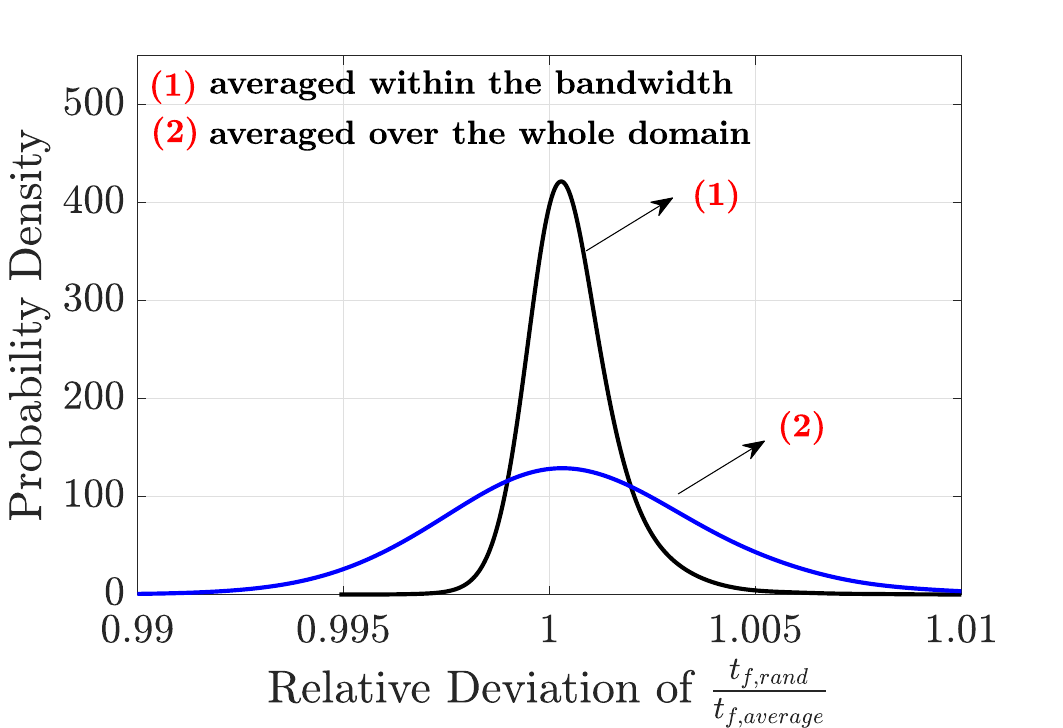}
		\label{fig:dplot}
	}
	\subfigure[]{
		\includegraphics[width=0.485\textwidth]{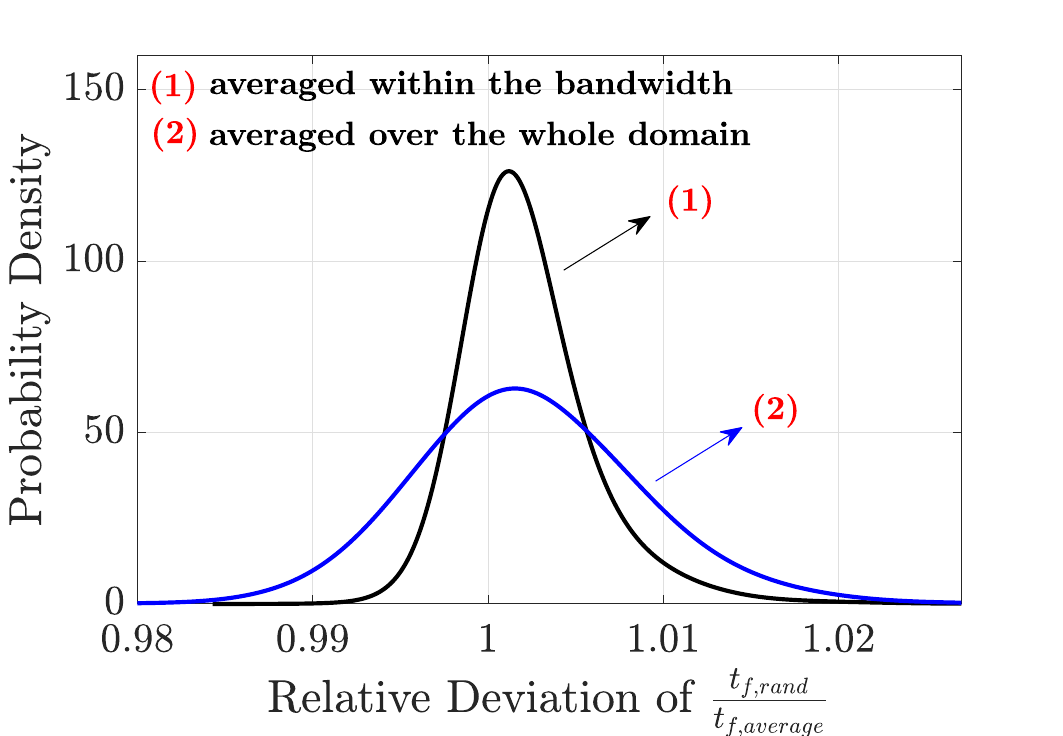}
		\label{fig:dplot}
	}
	\subfigure[]{
		\includegraphics[width=0.485\textwidth]{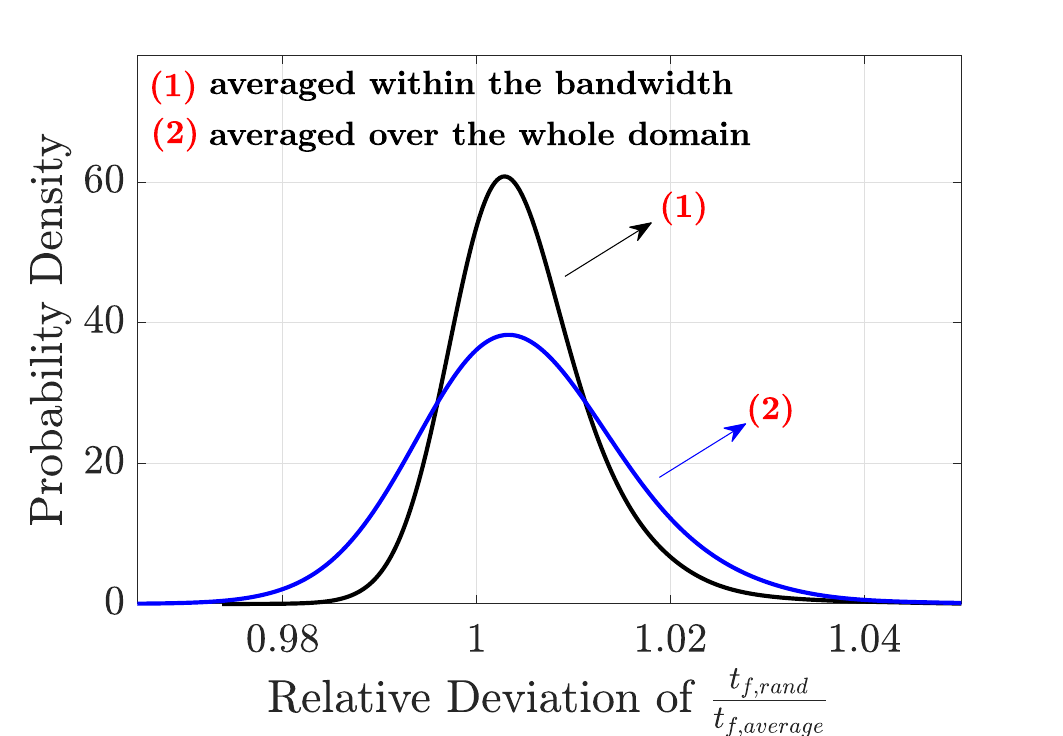}
		\label{fig:dplot}
	}
	\subfigure[]{
		\includegraphics[width=0.485\textwidth]{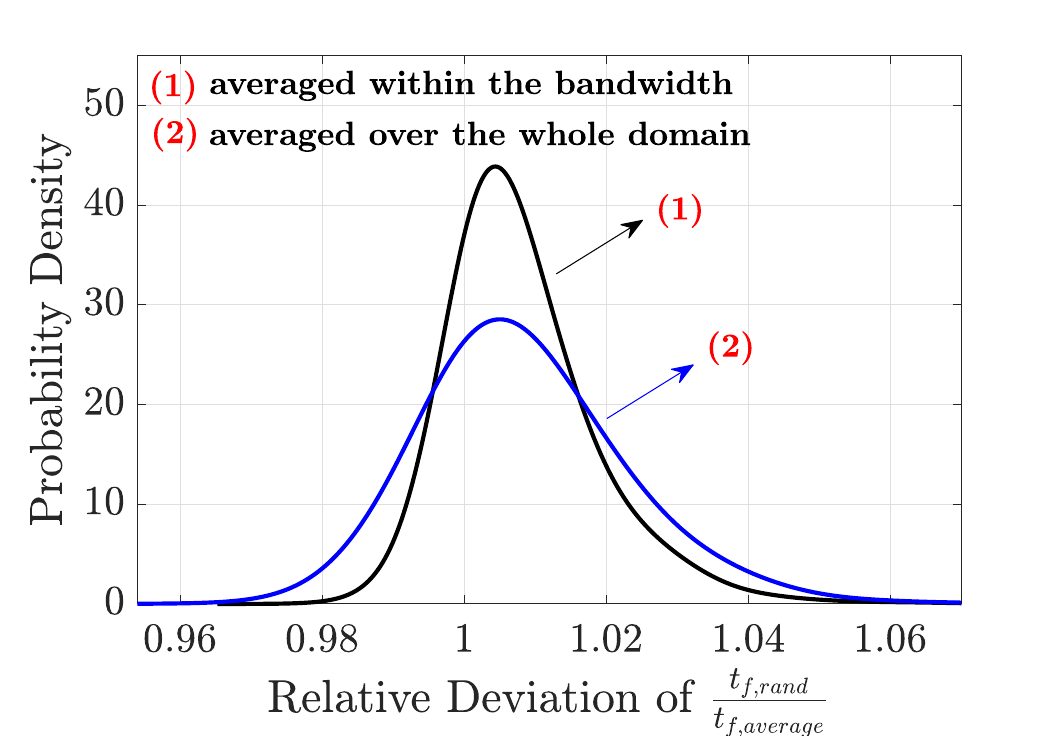}
		\label{fig:dplot}
	}
	\caption{Probability Distribution Functions (PDF): a) $p:=\frac{\max_{\Omega}(||\mathbf{W}||)}{v_0}=1/24$, b) $p:=\frac{\max_{\Omega}(||\mathbf{W}||)}{v_0}=2/24$, c) $p:=\frac{\max_{\Omega}(||\mathbf{W}||)}{v_0}=3/24$, d) $p:=\frac{\max_{\Omega}(||\mathbf{W}||)}{v_0}=4/24$}
	\label{fig:monte}
\end{figure}

\end{document}